\documentclass[11pt,a4paper, oneside]{article}

\usepackage[centertags]{amsmath}
\usepackage{amsfonts}
\usepackage{amssymb}
\usepackage{dsfont}
\usepackage{amsthm}
\usepackage{epsfig}
\usepackage[utf8]{inputenc} 
\usepackage{setspace}
\usepackage{ae} % Umlauts
\usepackage{eucal} 
\usepackage[authoryear, round]{natbib}
\usepackage{paralist}
\usepackage[usenames]{color}
\usepackage{setspace}
\usepackage{autonum}
\usepackage[active]{srcltx}
\usepackage{subcaption}

\usepackage{appendix}
\theoremstyle{plain}
\newtheorem{thm}{Theorem}[section]
\newtheorem{lem}[thm]{Lemma}

\theoremstyle{definition}
\newtheorem{defi}[thm]{Definition}
\newtheorem{remark}[thm]{Remark}
\newtheorem{algo}[thm]{Algorithm}

\theoremstyle{remark}
\newtheorem{Assumptions}[thm]{Assumptions}

\DeclareMathOperator{\sign}{sign}

\DeclareMathOperator*{\argmin}{argmin}

\DeclareMathOperator*{\arginf}{arginf}

\newcommand{\R}{\mathbb{R}}
\newcommand{\F}{\mathcal{F}}

\usepackage{xcolor}
\allowdisplaybreaks
\begin{document}

\title{A segment-wise dynamic programming algorithm for BSDEs}

\author{Christian Bender\footnote{Saarland University, Department of Mathematics,  Campus E2 4, D-66123  Saarbr\"ucken, Germany; Email: bender@math.uni-sb.de}\; and Steffen Meyer\footnote{Email: meyersteffen03@gmail.com}}

\maketitle
\begin{abstract}
We introduce and analyze a family of linear least-squares Monte Carlo schemes for backward SDEs, which interpolate between the one-step dynamic programming scheme of Lemor, Warin, and Gobet (Bernoulli, 2006) and the multi-step dynamic programming scheme of Gobet and Turkedjiev (Mathematics of Computation, 2016). Our algorithm approximates conditional expectations over segments of the time grid. We discuss the optimal choice of the segment length depending on the `smoothness' of the problem and show that, in typical situations, the complexity can be reduced compared to the state-of-the-art  multi-step dynamic programming scheme. 
\\[0.1cm]
\end{abstract}
\section{Introduction}

In this paper, we propose a family of regression-based algorithms for decoupled forward backward stochastic differential equations (FBSDEs) of the form
  \begin{eqnarray}\label{eq:FBSDE}
  	dX_t & =& b(t,X_t)dt+\sigma(t,X_t)dW_t,\quad
  	X_0 =x_0,\nonumber\\
  	-dY_t & =& f(t,X_t,Y_t,Z_t)dt-Z_tdW_t,\quad
  	Y_T=\xi(X_T)
  \end{eqnarray}
driven by a (multidimensional) Brownian motion $W$. Recall that the solution is a triplet $(X,Y,Z)$ of adapted processes, wherein $X$ and $Y$ describe the dynamics of the forward and backward SDE (BSDE), respectively, and $Z$ steers the backward SDE into the terminal condition without making use of future information.
 Among the numerous applications of BSDEs we mention (nonlinear) derivative pricing problems, drift control problems, and  stochastic representation formulas of Feynman-Kac type for semilinear parabolic partial differential equations, see, e.g., \cite{guyon, karoui1997,  pardoux_peng}. Loosely speaking, in these applications, the $Y$-part of the solution corresponds to the price process of the option, the value of the control problem, or the solution of the PDE, respectively, while the $Z$-part is required to represent the hedging portfolio, an optimal control, or the derivative of the PDE solution.
 
 Motivated by these and other applications, the literature on numerical schemes for BSDEs has tremendously increased over the last two decades, and we refer to \cite{kawai} for a comprehensive survey. Let us emphasize that the practical choice of a numerical scheme for the FBSDE \eqref{eq:FBSDE} crucially depends on the dimension $D$ of the state space of the diffusion process $X$. If $X$ takes values in a low dimensional space, the connection to PDEs via the four-step scheme \citep{ma_protter_yong} can be exploited and the application of fast PDE solvers leads to highly efficient schemes for the FBSDE, see, e.g., \cite{douglas_ma_protter, milstein_tretyakov, ma_shen_zhao}.  For moderate dimensions (up to about $D=10$), simulation based regression methods are among the most popular methods \citep{gobet_al, lemor2006, gobet2016}, while the recent multilevel Picard approach can tackle very high-dimensional problems \citep{E_al,hutzenthaler_al}.
 
 In the present paper, we contribute to the theory of simulation-based regression schemes. More precisely, we design a family of schemes which interpolate between the one-step-dynamic programming approach  (ODP) of \cite{lemor2006} and the multi-step dynamic programming approach (MDP) of \cite{gobet2016}. Both dynamic programming schemes rely on a backward recursion for the time discretization with nested conditional expectations which are approximated by empirical least-squares regression. The key difference is that the ODP computes the numerical solution of the BSDE at a time grid point $t_i$ by regressing a nonlinear function of the numerical solution at the previous time grid point $t_{i+1}$ only, while the MDP applies the numerical solutions at all grid points $t_{i+1}, t_{i+2},\ldots$ up to terminal time $T$. As argued in \cite{gobet2016} (see also \citealp{bender_denk}, for related considerations), the error propagation over time for the approximation of the $Y$-part of the solution can be significantly reduced in the MDP making it superior to the ODP in typical situations. A closer look at both schemes reveals, however, that the approximation of the $Y$-part is the dominating cost in the ODP, while the cost for  approximating the $Z$-part dominates the complexity of the MDP. In order to balance the cost for approximating $Y$-part and $Z$-part, we suggest a family of schemes, which iterates between a single ODP step and an MDP step on time segments containing $\lceil N^\alpha \rceil$ grid points -- here $N$ is the total number of grid points and $\alpha\in [0,1]$. We call this approach `segmentwise dynamic programming' (SDP) and analyze it in detail in this paper. In particular, we discuss how to  optimize the choice of the segment length parameter $\alpha$ in dependence on the state space dimension $D$ and on the smoothness of the PDE representation to the BSDE. It turns out that, in generic situations, the optimal choice of $\alpha$ is in the open interval $(0,1)$, demonstrating that neither the ODP ($\alpha=0$) nor the MDP ($\alpha=1$) are optimal. For the theoretical analysis, we assume that the forward diffusion $X$ can be perfectly simulated on the discrete time grids. One can, however, apply a higher order scheme for $X$ on the coarser grid to implement the SDP scheme, when this assumption is not satisfied.
 
 The paper is structured as follows: In Section 2 we introduce the setting and motivate the SDP algorithm, which is spelled out in detail in Section 3. The theoretical main results on convergence of the SDP scheme are presented in Section 4 along with a complexity analysis. The convergence behavior of the SDP scheme is illustrated and compared to the MDP scheme by a numerical experiment in Section 5.  Finally, the detailed error analysis (combining ideas from \citealp{gobet2016}, \citealp{lemor2006}, and \citealp{bender_denk}) is carried out in Section 6. Some proofs, which follow standard arguments, are provided in the online supplement.

\section{Setup and motivation}
Let $(\Omega,\mathcal{F},\mathbb{F},P)$ be a filtered probability where the filtration is the augmentation of the one generated by a $\mathcal{D}$-dimensional Brownian motion $W$. We consider a decoupled forward backward stochastic differential equation of the form
\begin{align}
dX_t & = b(t,X_t)dt+\sigma(t,X_t)dW_t,\quad
X_0 =x_0,\\
-dY_t & = f(t,X_t,Y_t,Z_t)dt-Z_tdW_t,\quad
Y_T=\xi(X_T)
\end{align}
with deterministic initial value $x_0\in \mathbb{R}^D$ and the terminal time $T>0$ under the following standing assumptions.
\begin{Assumptions} \label{assumptions}
\begin{itemize}
\item[]
\item[$(A_\xi)$] The function $\xi:\mathbb{R}^D\rightarrow \mathbb{R}$ is bounded by some constant $C_\xi$.
\item[$(A_L)$] The functions $b:[0,T]\times \mathbb{R}^D\rightarrow \mathbb{R}^D$, $\sigma:[0,T]\times \mathbb{R}^D\rightarrow \mathbb{R}^{D\times\mathcal{D}}$ and $f:[0,T]\times \mathbb{R}^D\times \mathbb{R}\times \mathbb{R}^\mathcal{D} \rightarrow \mathbb{R}$ are measurable, $\frac{1}{2}$-Hölder-continuous in the first variable and Lipschitz-continuous in the other variables, i.e.,  there exist constants $L_X$ and $L_f$ such that 
\begin{align}
|b(t,x)-b(t',x')|+|\sigma(t,x)-\sigma(t',x')| &\le L_X\left(|t-t'|^\frac{1}{2}+|x-x'|\right)\\
|f(t,x,y,z)-f(t',x',y',z')| &\le L_f\left(|t-t'|^\frac{1}{2}+|x-x'|+|y-y'|+|z-z'|\right)
\end{align}
for all $x,x'\in \mathbb{R}^D$, $t,t'\in [0,T]$, $y,y'\in \mathbb{R}$, $z,z'\in \mathbb{R}^\mathcal{D}$.
\item[$(A_f)$] The function $f$ is uniformly bounded by a constant $C_f$, i.e.,
\begin{align}
f(t,x,y,z)\le C_f
\end{align}
for all $t\in [0,T]$, $x\in \mathbb{R}^D$, $y\in \mathbb{R}$, $z\in \mathbb{R}^\mathcal{D}$.
\end{itemize}
\end{Assumptions}
\noindent
Assumption $(A_L)$ is standard for FBSDEs and yields important characteristics of the processes $X$, $Y$ and $Z$. For once, the assumption on $b$ and $\sigma$ ensures the existence of a unique strong solution $X$ of the SDE and that this solution satisfies $E[\sup_{t\in [0,T]}|X_t|^2]<\infty$ (see e.g. \citealp{karatzas2012}). Then, paired with the assumptions on $f$ and $\xi$ the solution of the BSDE can be expressed by deterministic functions of the SDE solution $X$, i.e., there exist deterministic functions $\overline{y}:[0,T]\times \mathbb{R}^D\rightarrow \mathbb{R}$ and $\overline{z}:[0,T]\times \mathbb{R}^D\rightarrow \mathbb{R}^\mathcal{D}$ such that $\overline{y}(t,X_t)=Y_t$ and $\overline{z}(t,X_t)=Z_t$ (see e.g. \citealp{karoui1997}). The boundedness conditions on $f$ and $\xi$ are posed for convenience only and could be relaxed with minor changes in the error analysis.\\
We now define the equidistant time grid 
\begin{align}
\pi:=\{t_i=i\Delta; ~i=0,\ldots, N\}
\end{align}
with step width $\Delta=T/N$ for a fixed $N\in \mathbb{N}$. We denote the increments of the Brownian motion $W$ on this time grid with $\Delta W_i$, i.e., $\Delta W_i:=W_{t_{i}}-W_{t_{i-1}}$, and define the functions
\begin{align}
\overline{q}^N_i(x)&:=E\left[Y_{t_{i+1}}|X_{t_i}=x\right],\\
\overline{z}^N_i(x)&:=E\left[\frac{\Delta W_{i+1}}{\Delta}Y_{t_{i+1}}\middle|X_{t_i}=x\right]
\end{align}
for all $i\in \{0,\ldots, N-1\}$. Then, under the standing assumptions, it holds that 
\begin{align}
\lim_{N\rightarrow \infty}\left(\max_{i=0,\ldots, N-1}E\left[|\overline{q}^N_i(X_{t_i})-Y_{t_i}|^2\right]+\sum_{i=0}^{N-1}E\left[\int_{t_i}^{t_{i+1}}|\overline{z}^N_i(X_{t_i})-Z_s|^2ds\right]\right)=0,
\end{align}
where the rate of convergence depends on the regularity of $Z$, see e.g. \cite{zhang}. The functions $\overline{q}_i^N$ and $\overline{z}^N_i$ can, therefore, be interpreted as time-discretized versions of the BSDE solution $(Y,Z)$. To obtain an  implementable approximation scheme for $\overline{q}_i^N$ and $\overline{z}_i^N$, fix any function
\begin{align}
\tau:\{0,\ldots, N-2\}\rightarrow \{1,\ldots, N-1\}
\end{align}
satisfying $\tau(i)\ge i+1$ for all $i\in \{0,\ldots, N-2\}$. Then the tower property of the conditional expectation and the Markov property of $X$ yield for any $i\in \{0,\ldots, N-2\}$
\begin{align}
\overline{q}_i^N(x)&=E[Y_{t_{i+1}}|X_{t_i}=x]\\
&=E\left[Y_{t_{\tau(i)+1}}+\int_{t_{i+1}}^{t_{\tau(i)+1}}f(t,X_t,Y_t,Z_t)dt-\int_{t_{i+1}}^{t_{\tau(i)+1}}Z_tdW_t\middle|X_{t_i}=x\right]\\
&=E\left[\overline{q}^N_{\tau(i)}(X_{\tau(i)})+\sum_{j=i+1}^{\tau(i)}\int_{t_{j}}^{t_{j+1}}f(t,X_t,Y_t,Z_t)dt\middle|X_{t_i}=x\right]\\
&\approx E\left[\overline{q}^N_{\tau(i)}(X_{\tau(i)})+\sum_{j=i+1}^{\tau(i)} f\left(t_j,X_{t_j},\overline{q}^N_{j}(X_{t_j}),\overline{z}^N_j(X_{t_j})\right)\Delta\middle|X_{t_i}=x\right]
\end{align}
and similarly
\begin{align}
\overline{z}_i^N(x)&=E\left[\frac{\Delta W_{i+1}}{\Delta}Y_{t_{i+1}}\middle|X_{t_i}=x\right]\\
&=E\left[\frac{\Delta W_{i+1}}{\Delta}\left(Y_{t_{\tau(i)+1}}+\int_{t_{i+1}}^{t_{\tau(i)+1}}f(t,X_t,Y_t,Z_t)dt-\int_{t_{i+1}}^{t_{\tau(i)+1}}Z_tdW_t\right)\middle|X_{t_i}=x\right]\\
&=E\left[\frac{\Delta W_{i+1}}{\Delta}\left(\overline{q}^N_{\tau(i)}(X_{\tau(i)})+\sum_{j=i+1}^{\tau(i)}\int_{t_{j}}^{t_{j+1}}f(t,X_t,Y_t,Z_t)dt\right)\middle|X_{t_i}=x\right]\\
&\approx E\left[\frac{\Delta W_{i+1}}{\Delta}\left(\overline{q}^N_{\tau(i)}(X_{\tau(i)})+\sum_{j=i+1}^{\tau(i)} f\left(t_j,X_{t_j},\overline{q}^N_{j}(X_{t_j}),\overline{z}^N_j(X_{t_j})\right)\Delta\right)\middle|X_{t_i}=x\right].
\end{align}
This motivates the time discretization scheme
\begin{eqnarray} \label{eq:1} \begin{split}
Q_{N-1}^N&:=E\left[\xi(X_{t_N})\middle|\mathcal{F}_{t_{N-1}}\right]\\
Z_{N-1}^N&:=E\left[\frac{\Delta W_{N}}{\Delta}\xi(X_{t_N})\middle|\mathcal{F}_{t_{N-1}}\right]\\
Q_{i}^{N}&:=E\left[Q^N_{\tau(i)}+\sum_{j=i+1}^{\tau(i)}f\left(t_j,X_{t_j},Q^N_{j},Z^N_{j}\right)\Delta\middle|\mathcal{F}_{t_i}\right],\quad i=N-2,\ldots, 0\\
 Z_{i}^{N}&:=E\left[\frac{\Delta W_{i+1}}{\Delta}\left(Q^N_{\tau(i)}+\sum_{j=i+1}^{\tau(i)}f\left(t_j,X_{t_j},Q^N_{j},Z^N_{j}\right)\Delta\right)\middle|\mathcal{F}_{t_i}\right],\quad i=N-2,\ldots, 0.
 \end{split} 
\end{eqnarray}
By the tower property of the conditional expectation, this definition of $Q^N_i$ and $Z^N_i$ does not depend on the choice of $\tau$ and is nothing but a reformulation of the backward Euler scheme by \cite{zhang, bouchard_touzi}. However, the conditional expectations cannot be calculated in closed form in general. Hence, when attempting to solve the BSDE, one has to replace the conditional expectations with some approximation operator resulting in different schemes depending on the choice of $\tau$. As our results will show, the choice of $\tau$ then influences both, the computational costs as well as the convergence properties.\\
The most natural choices for $\tau$ would be for once setting  $\tau(i)=i+1$ or $\tau(i)=N-1$ for all $i\in \{0,\ldots, N-2\}$. The first results in the classical one-step scheme of \cite{lemor2006} (to which we further refer to as ODP), the latter in the multi-step forward scheme (MDP for short) by \cite{gobet2016}. To understand the idea of the segment-wise dynamic programming algorithm that will be introduced in the next section, it is worth reviewing these two schemes and comparing the resulting algorithms.\\
Both algorithms work recursively backward in time by constructing estimates of the functions $\overline{q}^N_i$ and $\overline{z}^N_i$ through approximating the conditional expectations in the corresponding time discretization scheme via empirical (i.e., simulation-based) orthogonal projections on finite-dimensional function spaces, where the components $Q^N_j, Z_j^N$ with $j>i$ on the right-hand side of the discretization scheme are replaced by the approximations found in the previous steps of the recursion. As a result of the different schemes,   the approximation of the ODP algorithm depends at each time point $t_i$ only on the approximations at the time $t_{i+1}$ while in the MDP scheme, the approximation at each step $t_i$ depends on all the previously constructed ones, i.e., those at the time points from  $t_{i+1}$ up to $t_{N-1}$.  Since the approximations of $\overline{q}_i^N$ and $\overline{z}_i^N$ have to be evaluated, one has to simulate in each step of the MDP algorithm segments of the form $(X_{t_i}, X_{t_{i+1}},\ldots, X_{t_N})$ while it suffices in the ODP algorithm to simulate values of $X$ at just the current and the following time point. This obviously leads to higher simulation costs in the MDP scheme. However, since the algorithms recursively reuse the obtained approximations of $\overline{q}_i^N$ and $\overline{z}_i^N$ an error propagation between the time steps occurs. The error analysis in \cite{gobet2016} reveals that, from this perspective, the MDP scheme features improved convergence properties compared to the ODP.  \\
The idea of the segment-wise approach, which will be introduced in the next chapter, is to interpolate between the two  cases of the ODP and the MDP scheme in order to balance these two aspects, the computational costs and convergence properties.
\section{Segment-wise dynamic programming algorithm}
In this section, we present the segment-wise dynamic programming algorithm (SDP, for short) in detail. We first specify a  family of functions $\tau_\alpha$  to obtain the time discretization scheme that  interpolates between the ones from the ODP scheme and the MDP scheme via \eqref{eq:1}.  Then the algorithm is described in detail based on this discretization scheme.\\
For any $\alpha\in [0,1]$, consider the time grid 
\begin{align}
\overline{\pi}_\alpha:=\{(\Delta n \lceil N^\alpha \rceil )\wedge (T-\Delta); n\in \mathbb{N}\}
\end{align}
with step width $\lceil N^\alpha\rceil$ (up to a possibly smaller size in the last step), that consists of $\lceil N^{1-\alpha}\rceil$ time points at most. Based on these time grids define the functions 
\begin{align}
\tau_\alpha:\{0,1,\ldots, N-2\}\rightarrow \{1,\ldots, N-1\}
\end{align}
as
\begin{align}
\tau_\alpha(i):=\min \{ j>i:  j\Delta \in \overline{\pi}_\alpha\}.
\end{align}
For a fixed $\alpha$,  the choice $\tau=\tau_\alpha$ in \eqref{eq:1} then defines a discretization scheme where the time grid $\pi$ is separated in segments consisting of $\lceil N^{\alpha}\rceil$ points by the coarser time grid $\overline{\pi}_\alpha$. The resulting discretization scheme corresponds to an MDP scheme on each of these segments paired with a single step of an ODP scheme between consecutive segments connecting them. Moreover choosing $\alpha=0$ or $\alpha=1$ results in the classical ODP or MDP scheme respectively.\\
Now for a fixed $\alpha\in [0,1]$ the SDP algorithm works as follows.
\begin{algo}
\begin{itemize}
\item[]
\item Choose basis functions
\begin{align}
&p_{q,i}^{k}:\mathbb{R}^D\rightarrow \mathbb{R},\quad ~~k=1,\ldots, K_{q,i}\\
&p_{z,i}^{k}:\mathbb{R}^D\rightarrow \mathbb{R}^\mathcal{D},\quad k=1,\ldots, K_{z,i}
\end{align}
for each $i\in \{0,\ldots, N-1\}$ such that
\begin{align}
\sum_{i=0}^{N-1}\sum_{k=1}^{K_{q,i}}E\left[|p_{q,i}^k(X_{t_i})|^2\right]+\sum_{i=0}^{N-1}\sum_{k=1}^{K_{z,i}}E\left[|p_{z,i}^k(X_{t_i})|^2\right]<\infty.
\end{align}
Here the number of basis functions $K_{q,i}$, $K_{z,i}\in \mathbb{N}$ may depend on the time point $t_i$. We denote the function spaces spanned by these basis functions with $\mathcal{K}_{q,i}$ and $\mathcal{K}_{z,i}$ respectively, i.e.,
\begin{align}
\mathcal{K}_{q,i}&:=span \left(p_{q,i}^{1},\ldots, p_{q,i}^{K_{q,i}}\right)\\
\mathcal{K}_{z,i}&:=span \left(p_{z,i}^{1},\ldots, p_{z,i}^{K_{z,i}}\right).
\end{align}
The algorithm will approximate  $\overline{q}^N_i$ by empirical orthogonal projections on the subspaces $\mathcal{K}_{q,i}$ and $\overline{z}^N_i$ by projections on $\mathcal{K}_{z,i}$.
\item Initialize the algorithm by setting
\begin{align}
\Xi_{N-1}^{{N,M}}(x_N):=\xi(x_N)
\end{align}
for all $x_N\in \mathbb{R}^{D}$. Then, assuming $\Xi^{N,M}_i$ is already constructed, perform the following backward recursion for $i=N-1,N-2,\ldots, 0 $:
\begin{itemize}
\item [1*)]If $i=N-1$: Choose $M_{N-1}\in \mathbb{N}$, then simulate $M_{N-1}$ independent copies 
\begin{align}
\left(X^{[N-1,m,N]}_{t_{N-1}}, X^{[N-1,m,N]}_{t_N}, \Delta W^{[N-1,m,N]}_{N}\right)_{m=1,\ldots, M_{N-1}}
\end{align}
of the segment $(X_{t_{N-1}}, X_{t_N}, \Delta W_{N})$ and set 
\begin{align}
X^{[N-1,m,N]}:=X^{[N-1,m,N]}_{t_N}.
\end{align}
\item[1)] If $i<N-1$: Choose a  $M_i\in \mathbb{N}$,  then simulate $M_i$ independent copies 
\begin{align}
\left(X^{[i,m,N]}_{t_i},\ldots, X^{[i,m,N]}_{t_{\tau_\alpha(i)}}, \Delta W^{[i,m,N]}_{i+1}\right)_{m=1,\ldots, M_i}
\end{align}
of the segment $(X_{t_i},\ldots, X_{\tau_\alpha(i)}, \Delta W_{i+1})$ and set
\begin{align}
X^{[i,m,N]}:=\left(X^{[i,m,N]}_{t_{i+1}},\ldots, X^{[i,m,N]}_{t_{\tau_\alpha(i)}}\right).
\end{align}
\item[2)] Find solutions to the linear least-squares regression problems 
\begin{align}
\varphi_i^{q^{N,M}}=\argmin_{\psi \in \mathcal{K}_{q,i}}\left(\frac{1}{M}\sum_{m=1}^{M_i}\left|\psi\left(X^{[i,m,N]}_{t_i}\right)-\Xi^{N,M}_{i}\left(X^{[i,m,N]}\right)\right|^2\right)
\end{align}
and
\begin{align}
\varphi_i^{z^{N,M}}=\argmin_{\psi \in \mathcal{K}_{z,i}}\left(\frac{1}{M}\sum_{m=1}^{M_i}\left|\psi\left(X^{[i,m,N]}_{t_i}\right)-\frac{\Delta W_{i+1}^{[i,m,N]}}{\Delta}\Xi^{N,M}_{i}\left(X^{[i,m,N]}\right)\right|^2\right).
\end{align}
\item[3)] Define approximations $q_i^{N,M}$ and $z_i^{N,M}$ of the functions $\overline{q}^N_i$ and $\overline{z}^N_i$ via
\begin{align}
q_i^{N,M}:=\mathcal{T}_{C_{q,i}}\circ \varphi_i^{q^{N,M}}, \quad  z_i^{N,M}:=\mathcal{T}_{C_{z,i}}\circ \varphi_i^{z^{N,M}}
\end{align}
where $C_{q,i}:=C_\xi+(T-t_{i+1})C_f$ and $C_{z,i}:=\frac{C_{q,i}}{\Delta}$ are positive constants and $\mathcal{T}_c$ is the truncation function defined as 
\begin{align}
\mathcal{T}_c(x):=\sign(x)\min\{|x|,c\}
\end{align}
for any constant $c>0$ (acting componentwise on $\varphi^{z^{N,M}}$).
\item[ 4)] If $i\ge 1$, set 
\begin{align}
\Xi^{N,M}_{i-1}(x_{i},\ldots, x_{\tau_\alpha(i-1)}):=q_{\tau_\alpha(i-1)}^{N,M}(x_{\tau_\alpha(i-1)})+\sum_{j=i}^{\tau_\alpha(i-1)}\Delta  f\left(t_j,x_{j},q^{N,M}_j(x_j),z^{N,M}_j(x_j)\right)
\end{align}
as preparation for the next iteration.
\end{itemize}
\end{itemize}
\end{algo}
\noindent
The solutions to the empirical least squares problems can be computed numerically using a singular value decomposition. Hence, the algorithm is fully implementable as long as the segments $(X_{t_i},\ldots, X_{t_{\tau_\alpha(i)}})$ can be simulated. Then,  in the typical situation, for example $X=W$, the average costs for the simulation of one segment $(X_{t_i},\ldots, X_{t_{\tau_\alpha(i)}})$ are of order $\mathcal{O}(N^\alpha)$. For $\alpha<1$ this leads to smaller computation costs through simulations as the MDP scheme, where the average costs for simulating one set  $(X_{t_i},\ldots, X_{t_{N}})$ are of order $\mathcal{O}(N)$. When $X$ can not be sampled perfectly, it would, in principle,  be possible  to replace $X$ with some approximation scheme with minor changes in the error analysis. The problem is, however, to approximate $X$ in a way that sustains the gain in computation costs compared to the MDP scheme, which would not be the case  in the simplest approach when approximating $X$ with a naive Euler scheme starting at the time $0$. In theory, one could approximate $X_{t_i}$ with some high-order approximation scheme \citep{kloeden1992} and use an Euler scheme inside the segment $(X_{t_i},\ldots, X_{t_{\tau_\alpha(i)}})$ and preserve at least some gain in computation costs. However, we restrict the theoretical analysis to the assumption that the values of $X$ can be sampled directly on the time grid $\pi$.
\section{Convergence rates and analysis of the complexity}\label{sect:conv}
In this section, we state error bounds for the quadratic error of the SDP algorithm   We then analyze how to optimally calibrate  the parameters of the algorithm to the smoothness of the problem. The results show that the optimal choice of the segment length parameter $\alpha$ is always in the open interval $(0,1)$, and hence the SDP algorithm presented in Section 3 is advantageous when compared to the MDP and ODP schemes.

Our first bound for the mean-squared error of the SDP scheme is given in the following theorem:
\begin{thm}\label{thm:error1}
Under the standing assumptions, 
\begin{align}
&\max_{0\le i\le N-1}E\left[|q^{N,M}_i(X_{t_i})-\overline{q}^N_i(X_{t_i})|^2\right]+\sum_{i=0}^{N-1}\Delta E\left[|z_i^{N,M}(X_{t_i})-\overline{z}_i^N(X_{t_i})|^2\right]\\
&\le c\max_{i\in \mathcal{I}}\left(N^{1-\alpha} \inf_{\psi \in \mathcal{K}_{q,i}} E\left[|\psi(X_{t_i})-\overline{q}^N_i(X_{t_i})|^2\right] + N^{2-2\alpha}\frac{K_{q,i}}{M_i}+ N^{2-2\alpha}\frac{K_{q,i}\log(M_i)}{M_i}\right)\\
&\quad +c\max_{0\le i\le N-1}\left(\inf_{\psi \in \mathcal{K}_{q,i}} E\left[|\psi(X_{t_i})-\overline{q}^N_i(X_{t_i})|^2\right] + \inf_{\psi \in \mathcal{K}_{z,i}} E\left[|\psi(X_{t_i})-\overline{z}^N_i(X_{t_i})|^2\right]\right.\\
&\quad+\left.\frac{K_{q,i}}{M_i}+N\frac{K_{z,i}}{M_i}+\frac{K_{q,i}\log(M_i)}{M_i}+ N \frac{K_{z,i}\log(M_i)}{M_i}\right)\\
&  \quad+cN\mathcal{R}^N
\end{align}
where $\mathcal{I}:=\{i: t_i\in \overline{\pi}_\alpha\}$, $c$ is a positive constant not depending on $N$ and
\begin{align}
\mathcal{R}^N:=\sum_{i=1}^{N-1}E\left[\left(\int_{t_{i}}^{t_{i+1}}E\left[f\left(s,X_s,Y_s,Z_s\right)-f\left(t_i,X_{t_i},\overline{q}_i^N(X_{t_i}),\overline{z}^N_i(X_{t_i})\right)\middle|X_{t_{i-1}}\right]ds\right)^2\right].
\end{align}
\end{thm}
\noindent
As usual, we can think of this bound as a composition of terms due to three different error sources. The first source of error is due to the projection on the finite-dimensional subspaces. This projection error can be controlled by the choice of basis functions. The term $\mathcal{R}^N$ only depends on the continuous-time  BSDE solution and the finer time grid (through the discretized functions $\overline{q}^N_i$ and $\overline{z}^N_i$), but not on the approximation obtained by the algorithm. It can be interpreted as part of the time discretization error. The remaining terms are statistical error terms depending on the sample size. Those can be controlled by increasing the number of simulations in dependence of the number of basis functions.

The bound shows the influence of the parameter $\alpha$, as the projection and statistical error terms regarding  $\overline{q}^N$  appear once at all time steps and once  on the time steps of the coarser time grid $\overline{\pi}$ with different factors that are  decreasing in $\alpha$. Although larger values for $\alpha$ appear to be favorable from this perspective,  increasing the parameter $\alpha$ also results in higher computational costs as paths then have to be simulated on larger segments. Hence, when optimizing $\alpha$ a trade-off between convergence properties and computational costs has to be taken into account.

For $\alpha=1$ the terms in the first bracket are dominated by the remaining terms, and we essentially reproduce the error analysis of the MDP scheme in \cite{gobet2016} with some minor differences: In particular, the projection error in Theorem \ref{thm:error1} is formulated in terms of the true continuous time solution $(Y,Z)$ of the BSDE via the functions $\overline{q}^N_i$ and $\overline{z}_i^N$ while it is stated in terms of the backward Euler discretization scheme for BSDEs in \cite{gobet2016}.
When choosing $\alpha=0$, one ends up with an error analysis for the ODP scheme with independent re-simulation at every time point. This differs from the error analysis in \cite{lemor2006}, as they only apply  one cloud of simulations, which is re-used  at any time point, resulting  in an additional (and dominating) interdependency error. Hence, the bound obtained with Theorem \ref{thm:error1} for $\alpha=0$ also allows for a comparison of the ODP and MDP algorithms in a unified setting.

Typical bounds for the term $\mathcal{R}^N$ in dependence of $N$ are of order $N^{-2}$ or $N^{-3}$, depending on   regularity assumptions as illustrated by the following theorems. These theorems can be derived from Theorem \ref{thm:error1} by applying standard techniques to bound the expected quadratic difference between the functions $\overline{q}^N$ and $\overline{z}^N$ and their continuous time counterparts  $\overline{y}$ and $\overline{z}$ depending on the regularity of the BSDE, see \cite{zhang, gobet_labart}. Detailed proofs are provided in an online supplement\footnote{available at: https://www.uni-saarland.de/fakultaet-mi/stochastik/prof-dr-christian-bender/publications.html}.
\begin{thm}\label{thm:error2}
Under the standing assumptions, suppose that $\overline{z}$ is $\frac{1}{2}$-Hölder continuous in $t$ and Lipschitz continuous in $x$. Then,
\begin{align}
	& \max_{0\le i\le N-1}E\left[|q^{N,M}_i(X_{t_i})-\overline{y}(t_i,X_{t_i})|^2\right]+\sum_{i=0}^{N-1}E\left[\int_{t_i}^{t_{i+1}}|z_i^{N,M}(X_{t_i})-\overline{z}(s,X_s)|^2ds\right]\\
	& \le c\max_{i\in \mathcal{I}}\left(N^{1-\alpha} \inf_{\psi \in \mathcal{K}_{q,i}} E\left[|\psi(X_{t_i})-\overline{y}(t_i,X_{t_i})|^2\right] + N^{2-2\alpha}\frac{K_{q,i}}{M_i}+ N^{2-2\alpha}\frac{K_{q,i}\log(M_i)}{M_i}\right)\\
	&\quad +c\max_{0\le i\le N-1}\left(\inf_{\psi \in \mathcal{K}_{q,i}} E\left[|\psi(X_{t_i})-\overline{y}(t_i,X_{t_i})|^2\right] + \inf_{\psi \in \mathcal{K}_{z,i}} E\left[|\psi(X_{t_i})-\overline{z}(t_i,X_{t_i})|^2\right]\right.\\
	& +\left.\frac{K_{q,i}}{M_i}+N\frac{K_{z,i}}{M_i}+\frac{K_{q,i}\log(M_i)}{M_i}+ N \frac{K_{z,i}\log(M_i)}{M_i}\right)\\
	&  +cN^{-1},
\end{align}
where $\mathcal{I}:=\{i: t_i\in \overline{\pi}_\alpha\}$ and $c$ is a positive constant not depending on $N$.
\end{thm}
\noindent
Theorem \ref{thm:error2} provides, in the context of Theorem \ref{thm:error1}, the standard error bound of order $N^{-1/2}$ for the time discretization within the backward Euler scheme \citep{zhang,bouchard_touzi}.
Note that the projection errors are measured with respect to the continuous time PDE representation of the BSDE \eqref{eq:FBSDE}. Hence the choice of the basis functions can be adapted to known regularity results for the corresponding semilinear parabolic Cauchy problem, see, e.g.,
\cite{crisan_delarue}.

 Under additional regularity  assumptions, it is known from \cite{gobet_labart} that a time discretization error of order $N^{-1}$ can be achieved. In the context of the SDP, the following result can be derived in this respect.
\begin{thm}\label{thm:error3}
Additionally to the standing assumptions, suppose that $X=W$, and that $f $ and $\overline{y}$ are  twice, resp. $s+1$ times continuously differentiable ($s\ge 2$) with bounded derivatives. Then,
\begin{align}
& \max_{0\le i\le N-1}E\left[|q^{N,M}_i(W_{t_i})-\overline{y}(t_i,W_{t_i})|^2\right]+\sum_{i=0}^{N-1}\Delta E\left[|z^{N,M}_i(W_{t_i})-\overline{z}(t_i,W_{t_i})|^2\right]\\
& \le c\max_{i\in \mathcal{I}}\left(N^{1-\alpha} \inf_{\psi \in \mathcal{K}_{q,i}} E\left[|\psi(W_{t_i})-\overline{q}^{N}_i(W_{t_i})|^2\right] + N^{2-2\alpha}\frac{K_{q,i}}{M_i}+ N^{2-2\alpha}\frac{K_{q,i}\log(M_i)}{M_i}\right)\\
& \quad +c\max_{0\le i\le N-1}\left(\inf_{\psi \in \mathcal{K}_{q,i}} E\left[|\psi(W_{t_i})-\overline{q}^{N}_i(W_{t_i})|^2\right] + \inf_{\psi \in \mathcal{K}_{z,i}} E\left[|\psi(W_{t_i})-\overline{z}^{N}_i(W_{t_i})|^2\right]\right.\\
& +\left.\frac{K_{q,i}}{M_i}+N\frac{K_{z,i}}{M_i}+\frac{K_{q,i}\log(M_i)}{M_i}+ N \frac{K_{z,i}\log(M_i)}{M_i}\right)\\
&  +cN^{-2}
\end{align}
where $\mathcal{I}:=\{i: t_i\in \overline{\pi}_\alpha\}$ and $c$ is a positive constant not depending on $N$. Furthermore, $\overline{q}^N$ and $\overline{z}^N$ are bounded and $s+1$ times,  respectively,  $s$ times continuously differentiable with bounded derivatives.
\end{thm}
\noindent

Before we prove the theorems presented above in Section 6, we analyze how to optimally choose $\alpha$ and derive the resulting complexity of the algorithm. To this end, we first calibrate the algorithm in dependence of $\alpha$ to achieve a squared error of the order $N^{-2\theta}$
with $\theta \in \{1/2,1\}$  with local polynomials as basis functions, cp. \cite{lemor2006,gobet2016}.

For simplicity we use the same approximation space $\mathcal{K}_z$ for the approximation of $\overline{z}^N$ in each time step. Due to the additional projection error on the coarser time grid $\overline{\pi}_\alpha$  we distinguish between the time points inside and outside $\overline{\pi}_\alpha$ for the approximation of $\overline{q}^N$. We assume that the same  approximation space $\overline{\mathcal{K}}_q$ is used for all time points $t_i\in \overline{\pi}_\alpha$ while we use a possibly different approximation space $\mathcal{K}_q$ at the other time points. Analogously we assume that we use $M_i=M$ simulations  at each time point $t_i$ outside the coarser time grid $\overline{\pi}$ and $M_i=\overline{M}$ simulations otherwise. As already mentioned, we  assume that one segment $(X_i,\ldots, X_{\tau_\alpha(i)})$ of a path of $X$ can be simulated at the cost of $N^\alpha$.  Under these assumptions, assuming we can evaluate the driver and the basis functions at cost 1, the performed simulations and evaluations during the algorithm lead to costs of order  
\begin{align}
NN^\alpha M+ N^{1-\alpha} N^\alpha\overline{M}.
\end{align}
As basis functions we take local polynomials on cubes which we choose disjoint such that their union  contains the set $\{x\in \R^D: |x|<C_b\}$ for a constant $C_b>0$. We suppose that the edge length of the cubes is  $\overline{\delta}_q$ for the approximation of $\overline{q}^N$ on the time grid $\overline{\pi}_\alpha$, $\delta_q$ for the approximation of $\overline{q}^N$ at all other time points and  $\delta_z$ for the approximation of $\overline{z}^N$. Assuming that $\overline{y}$ and $\overline{z}$ are  $s+1$ times, respectively, $s$  times continuously differentiable with bounded derivatives we set the degree of the polynomials as $s$ for the approximation of $\overline{q}^N$ and $s-1$ for $\overline{z}^N$. We denote  the set of polynomials of degree less than or equal to $l$ by $P_l$. Then, the projection error in the context of 
Theorem \ref{thm:error2} can be estimated by a Taylor expansion on each cube:

\begin{eqnarray}\label{eq:2}
\inf_{\psi\in \mathcal{K}_q}E\left[|\psi(X_{t_i})-\overline{y}(t_i,X_{t_i})|^2\right]&\le & E\left[|\overline{y}(t_i,X_{t_i})|^2\mathds{1}_{|X_{t_i}|_>C_b}\right] \nonumber \\
&&+\sum_{H\in \mathcal{H}_{q}}\inf_{\psi\in P_{s}}E\left[|\psi(X_{t_i})-\overline{y}(t_i,X_{t_i})|^2\mathds{1}_{X_{t_i}\in H}\right]\\
&\le& \|\overline{y}(t_i,.)\|^2_\infty P\left(|X_{t_i}|>C_b\right)+c\|\overline{y}(t_i,.)^{(s+1)}\|_\infty^2(\delta_q^{s+1})^2, \nonumber
\end{eqnarray}
where $\mathcal{H}_q$ denotes the collection of cubes applied for the approximation of $\overline{q}^N$ at all time points outside $\overline{\pi}$. Under the assumption that $\sup_{0\le i\le N}E[e^{\varpi|X_{t_i}|}]<\infty$ for some $\varpi>0$, it follows by the Markov inequality that the choice  $C_b=2\theta \varpi^{-1}\log(N+1)$ ensures that the first term in \eqref{eq:2} is of order $N^{-2\theta}$. The same bound holds for the second term when choosing the edge length of the hypercubes as $\delta_q=cN^{-\frac{\theta}{s+1}}$. Therefore  it suffices to choose  $K_{q}$ of the order $N^{D\frac{\theta}{s+1}}\log^D(N+1)$ to ensure that
\begin{align}
\inf_{\psi\in \mathcal{K}_{q}}E[|\psi(X_i)-\overline{y}(t_i,X_{t_i})|^2]\in O(N^{-2\theta}). 
\end{align}
Following the same argumentation, we set 
\begin{align}
\overline{K}_{q}= cN^{D\frac{\theta+\frac{1-\alpha}{2}}{s+1}}\log^D(N+1),\quad  K_{z}= cN^{D\frac{\theta}{s}}\log^D(N+1)
\end{align}
for a positive constant $c$ to ensure 
\begin{align}
N^{1-\alpha}\inf_{\psi\in \overline{\mathcal{K}}_{q}}E\left[|\psi(X_i)-\overline{y}(t_i,X_{t_i})|^2\right]\in O(N^{-2\theta}), \quad 
\inf_{\psi\in \mathcal{K}_{z}}E\left[|\psi(X_i)-\overline{z}(t_i,X_{t_i})|^2\right]\in O(N^{-2\theta})
\end{align}
where the change from $s+1$ to $s$ in the number of basis functions in the approximation of $\overline{z}^N$ occurs due to the lower smoothness of $\overline{z}^N$ and $\overline{K}_{q}$ and $K_z$ denote the number of basis functions of the space $\overline{\mathcal{K}}_{q}$ and $\mathcal{K}_z$ respectively. 
The same argument applies in the context   of Theorem \ref{thm:error3}, replacing $\overline{y}$ and $\overline{z}$ by their discrete counterparts $\overline{q}^N$ and $\overline{z}^N$. 

Given the size of the approximation spaces we hence have to choose $M$ of the order
\begin{align}
N^{2\theta}\max\{K_q,NK_z\}=N^{1+2\theta}K_z
\end{align}
and $\overline{M}$ of the order
\begin{align}
N^{2\theta}\max\{N^{2-2\alpha}\overline{K}_q,NK_z\}
\end{align}
in order to bound the statistical error terms asymptotically by a multiple of $N^{-2\theta}$ (assuming that the driver $f$ is not independent of $Z$ and ignoring log-factors from now on). Then, in dependence of $\alpha$, the  computation costs of the algorithm grows as
\begin{align}
\mathcal{C}=\max\{N^{1+\alpha}M,N\overline{M}\} =:\max\{\mathcal{C}_\pi,\mathcal{C}_{\overline{\pi}}\}.
\end{align}
Here $\mathcal{C}_\pi$ is increasing in $\alpha$ and $\mathcal{C}_{\overline{\pi}}$ is decreasing in $\alpha$. The optimal choice of $\alpha$ is, thus, obtained by equating both terms, leading to
\begin{align}
\alpha_{opt}=\frac{\frac{1}{2}-\frac{\theta}{s}+\frac{s+1}{D}}{\frac{1}{2}+\frac{3(s+1)}{D}}.
\end{align}
It always lies in the open interval $(0,1)$, provided  $\vartheta=\frac{1}{2}$, $s\ge 1$ or $\vartheta=1$, $s\ge 2$. The resulting computational costs are of the order 
\begin{align}
\mathcal{C}=NN^{\alpha_{opt}} M & = N^{2+2\theta+D\frac{\theta}{s}+\frac{\frac{1}{2}-\frac{\theta}{s}+\frac{s+1}{D}}{\frac{1}{2}+\frac{3(s+1)}{D}}} = N^{3+2\theta+D\frac{\theta}{s}-\frac{\frac{\theta}{s}+\frac{2(s+1)}{D}}{\frac{1}{2}+\frac{3(s+1)}{D}}}.
\end{align}
In comparison, choosing $\alpha=1$ (the MDP case), the computational costs are of the order
\begin{align}
\mathcal{C}=cN^{3+2\theta+D\frac{\theta}{s}},
\end{align}
as already shown in \cite{gobet2016}. Hence, the suggested SDP algorithm with the optimal choice of the segment length reduces the complexity by a factor of the order 
\begin{align}
N^{1-\alpha_{opt}}=N^{\frac{\theta/s+2(s+1)/D}{1/2+3(s+1)/D}}
\end{align}
depending on the smoothness $s$ of the problems and the dimension $D$ of the state space of $X$. 
\section{Numerical example}
In this section, we  compare the SDP and MDP approach in a numerical test case in order to illustrate our theoretical results.\\
For this purpose, we define for each $x\in \R^D$ and all $t\in [0,0.2]$ the function
\begin{align}
\varphi(t,x):=\exp\left(-\sum_{d=1}^D|x^{(d)}-t|^{0.3}\right)\sum_{d=1}^D\left(x^{(d)}-t\right)^2
\end{align}
and for $d=1,\ldots, D$
\begin{align} 
\phi_d(t,x):=\exp\left(-\sum_{d=1}^D|x^{(d)}-T|^{0.3}\right)\left(x^{(d)}-t\right)\left(2-0.3|x^{(d)}-t|^{0.3}\right)\sum_{e=1, e\neq d}^D\left(x^{(d)}-t\right)^2.
\end{align}
We then consider the BSDE driven by a Brownian motion $X=W$ with terminal time $T=0.2$, the terminal condition 
\begin{align}
\xi(W_T)= \varphi(T,W_T)
\end{align}
and driver
\begin{align}
f(t,x,y,z):=\min\{|z|,c\}-|\nabla \varphi(t,x)|-(\partial_t+\frac{1}{2}\triangle)\varphi(t,x).
\end{align}
for $c:=\sup_{(t,x)\in [0,T]\times\mathbb{R}^D} |\nabla \varphi(t,x)|$, noting that $\varphi$ is bounded and twice continuously differentiable with bounded derivatives.
It can be easily checked by  It\^o's formula that the analytic solution to this BSDE is given by
\begin{align}
Y_t&=\varphi(t,W_t)\\
Z^{(d)}_t&=\frac{\partial}{\partial x^{(d)}}\varphi(t,W_t)=\phi_d(t,W_t)\quad d=1,\ldots, D,
\end{align}
and that the standing assumptions are satisfied.

For a comparison between the MDP algorithm and the SDP algorithm,  we calibrate the function basis and the number of sample paths as functions of the number of time steps $N$ in line with the complexity analysis of the previous section with $\vartheta=1/2$ and $s=1$. We increase the number of time steps $N$ and re-run each algorithm 40 times in dimension $D=2$. Below we report the mean-squared errors and the average run times of both algorithms across the 40 repetitions.  More precisely, we proceed in the following way:
\\[0.2cm]
\emph{Calibration:}\\
We use piecewise linear functions for the approximation of $\overline{q}^N_i$ and piecewise constant functions for the one of $\overline{z}^N_i$ in both algorithms. At each time, we set the outer bound of the hypercubes as a multiple of the standard deviation of the Brownian motion at that time. More precisely, we choose  $C_{b,i}:=(2\log(N)+2)\sqrt{t_i}$ at the time $t_i$.  
As edge length of the cubes we choose $\delta_z=\sqrt{T}/N^{\frac{1}{2}}$ and $\delta_q=\sqrt{T}/N^{\frac{1}{4}}$ for the MDP algorithm leading  to $K_{z,i}\in \mathcal{O}(\lceil R_i/\delta_z\rceil^D)=\mathcal{O}(N\log^2(N))$ and $K_{q,i}\in \mathcal{O}(\lceil R_i/\delta_q\rceil^D)=\mathcal{O}(N^{1/2}\log^2(N))$ at time $t_i$. We re-simulate the sample paths for the approximation of each conditional expectation and use $M_{q,i}=10NK_{q,i}\in\mathcal{O}(N^{3/2}\log^2(N))$ simulations for $Q$ at time $t_i$ and $M_{z,i}=5N^2K_{z,i}\in\mathcal{O}(N^{3}\log^2(N)) $ simulations for $Z$.\\
We may choose $\delta_z$, $K_{z,i}$ and $M_{z,i}$ for the SDP algorithm in the same way as in the MDP algorithm as well as $\delta_q$ and $K_{q,i}$ and $M_{q,i}$, if $t_i\not \in \overline{\pi}$. The  optimal value of the segment length parameter is $\alpha_{opt}=\frac{2}{7}$. For the time points in $\overline{\pi}$ we  choose $\overline{\delta_q}=1.5\sqrt{T}/N^{-\frac{1}{2}+\frac{\alpha}{4}}$ leading to $\overline{K}_{q,i}\in \mathcal{O}(N^{1-1/7}\log^2(N))$. As number of simulations for the approximation of $Q$ at these time points we choose $\overline M_{q,i}=N^{3-2\alpha}\overline{K}_{q,i}\in \mathcal{O}(N^{3+2/7}\log^2(N))$.
\\[0.2cm]
\emph{Measuring the errors:}\\
We measure the mean-squared error (MSE) of both algorithms for the approximation of $\overline{q}^N_i$ and $\overline{z}^N_i$ separately and consider the average error over all the time steps. Additionally, we report the approximation error of $Y$ at time zero, which is a relevant quantity in many applications such as nonlinear option pricing. Precisely, we consider the average over the 40 repetitions of the three indicators
\begin{align}
\mathcal{C}_{y,av}&:=\frac{1}{N}\sum_{i=1}^{N-1}\sum_{H\in \mathcal{H}_{q,i}}\big|q^{M,N}_i(\Theta_{H})-\overline{y}(t_i,\Theta_{H})\big|^2P\big(X_{t_i}\in H\big)\\
\mathcal{C}_{y,0}&:=\big|q^{M,N}_0(0)-\overline{y}(0,0)\big|^2\\
\mathcal{C}_{z,av}&:=\frac{1}{N}\sum_{i=0}^{N-1}\sum_{H\in \mathcal{H}_{z,i}}\big|z^{M,N}_i(\Theta_{H})-\overline{z}(t_i,\Theta_{H})\big|^2P\big(X_{t_i}\in H\big)
\end{align}
where $\Theta_{H}$ is the center of the cube $H$, and $\mathcal{H}_{q,i}$ and $\mathcal{H}_{z,i}$ are the sets of  cubes used at time $t_i$ for the approximation of $\overline{q}^N_i$ and $\overline{z}^N_i$, respectively. We refer to the arithmetic mean of $\mathcal{C}_{y,av}$ and $\mathcal{C}_{z,av}$ over the 40 repetitions as average mean-squared error in time and analogously to the arithmetic mean of $\mathcal{C}_{y,0}$ over the 40 repetitions as mean-squared error  at time 0.

\begin{figure}[h!]
\begin{subfigure}{.5\textwidth}
  \centering
  \includegraphics[width=.8\linewidth]{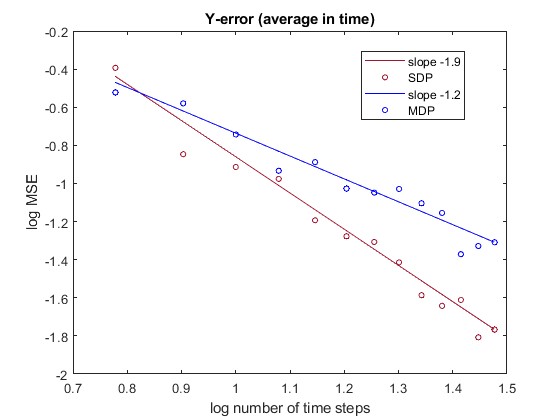}
  \caption{average mean-squared error in time for $Y$.}
  \label{fig:sfig1}
\end{subfigure}%
\begin{subfigure}{.5\textwidth}
  \centering
  \includegraphics[width=.8\linewidth]{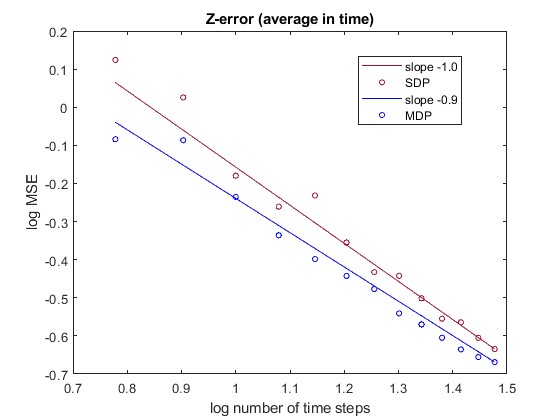}
\caption{average mean-squared error in time for $Z$.}
  \label{fig:sfig2}
\end{subfigure}
\\
\begin{subfigure}{.5\textwidth}
  \centering
  \includegraphics[width=.8\linewidth]{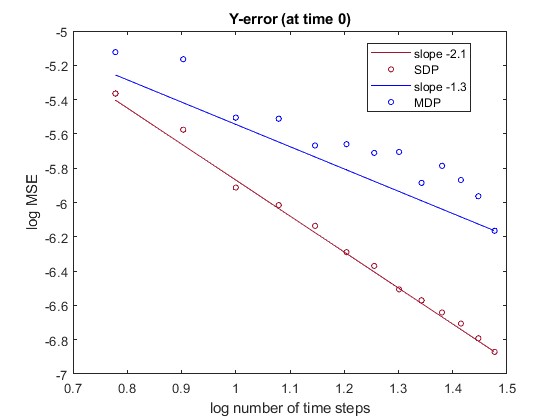}
  \caption{mean-squared error for $Y$ at 0.}
  \label{fig:sfig3}
\end{subfigure}%
\begin{subfigure}{.5\textwidth}
  \centering
  \includegraphics[width=.8\linewidth]{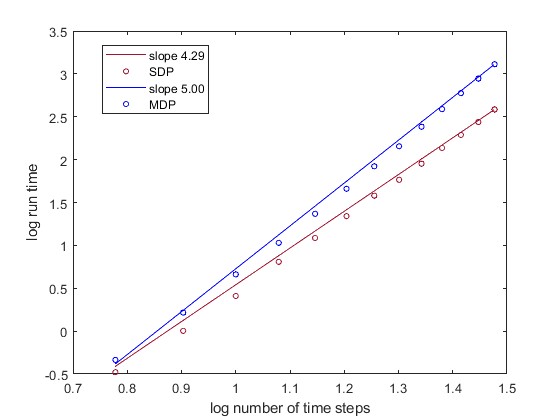}
  \caption{average run time.}
  \label{fig:sfig4}
\end{subfigure}
\caption{Numerical results in dimension $D=2$.}\label{fig:fig}

\end{figure}
\ \\[0.1cm]
\emph{Numerical results:}
\\
Figure \ref{fig:fig} depicts $\log_{10}$-$\log_{10}$-plots of the mean-squared errors (as introduced above) and the average run times of both algorithms as the number $N$ of time discretization points increases from $N=6$ to $N=30$. In view of the results in Theorem \ref{thm:error2} and the calibration of the algorithms we expect that the mean-squared errors decrease at the order $N^{-1}$. The numerical values of the three error indicators in Figures \ref{fig:sfig1}--\ref{fig:sfig3} (blue circles) for the MDP scheme are largely in accordance with this expected convergence behavior. The numerical results for the SDP scheme (red circles) feature a convergence rate of $N^{-1}$ for the average mean-squared error of the $Z$-component (Figure \ref{fig:sfig2}) but a faster convergence rate of about $N^{-2}$ for the mean-squared errors of the $Y$-components (Figures \ref{fig:sfig1} and \ref{fig:sfig3}). A possible explanation of this excellent empirical convergence rate are potential variance benefits when running empirical regressions on shorter time segments. Figures \ref{fig:sfig1} and \ref{fig:sfig3} also illustrate a higher variance of the $Y$-approximations of the MDP scheme across the repetitions compared to the SDP scheme. Finally, the average run time over the repetitions are plotted against the number of time discretization points in Figure \ref{fig:sfig4}. As before, the blue circles and red circles mark the numerical results for the MDP algorithm and the SDP algorithm, respectively. The solid lines correspond to the expected run time behavior of $N^5$ for the MDP scheme and $N^{4+2/7}$ for the SDP scheme derived from the complexity analysis at the end of Section~\ref{sect:conv}. Taking into account that log-factors, which have been neglected throughout the complexity analysis, play a role for small values of $N$ in both algorithms, Figure \ref{fig:sfig4} confirms the theoretical run time benefits of the SDP scheme in the practical implementation.

Note that this example is only supposed to serve as a proof of concept of the improved complexity of the SDP scheme compared to the MDP scheme. Parallelization and variance reduction techniques can be incorporated analogously to the MDP case and are of prime importance when implementing regression on local polynomials for higher dimensional problems (of around $D=10$), see, e.g., \cite{Gobet_al_16}.

\section{Error analysis}
In this section, we present a complete and detailed error analysis of the SDP algorithm. In the first subsection, we introduce additional notation used in the proofs and present some key tools for the error analysis. The following subsections are then dedicated to the derivation of error bounds for the approximation of  $\overline{q}^N$ and $\overline{z}^N$ respectively.  We will establish a sort of recursion formula for both parts, allowing us to bound the quadratic error at the time $t_i$ by the one at time $t_ {i+1}$ plus an additional driver-dependent term. This illustrates the error propagation between the time steps. We will then derive global bounds for the quadratic error for both approximations before we analyze the driver-dependent terms appearing in both recursions further in Subsection 6.5. Finally, in the last subsection, we combine the obtained bounds to derive the result presented in Theorem \ref{thm:error1}.
\subsection{Preliminaries and key tools}
In this section, we discuss some ramifications for the error analysis by introducing additional objects and presenting key tools for the analysis. \\
First, recall the definition of the functions $\overline{q}^N_i$ and $\overline{z}^N_i$. By setting
\begin{align}
\Xi^N_{N-1}(\underline{x}_{N-1})&:=\xi(x_N) \\
\Xi^N_{i}(\underline{x}_i)&:=\overline{q}^N_{\tau_\alpha(i)}(x_{\tau_\alpha(i)})+\int_{t_{i+1}}^{t_{\tau_\alpha(i)+1}}f(s,x_s,\overline{y}(s,x_s),\overline{z}(s,x_s))ds  &i\in \{0,\ldots, N-2\}
\end{align}
for any $\underline{x}_i:=(x_s)_{t_i\le s\le T}\in (\mathbb{R}^{D})^{[t_i,T]}$, it holds
\begin{align}
\overline{q}^N_i(x)=E\left[\Xi^N_i\left((X_s)_{t_i\le s\le T}\right)\middle|X_{t_i}=x\right]
\end{align}
and
\begin{align}
\overline{z}^N_i(x)=E\left[\frac{\Delta W_{i+1}}{\Delta}\Xi^N_i\left((X_s)_{t_i\le s\le T}\right)\middle|X_{t_i}=x\right].
\end{align}
Note that $\Xi^N_i$ differs from $\Xi^{N,M}_i$ in two ways: in the functions $\Xi^{N,M}_i$, the true solution $(Y,Z)$ of the BSDE  is replaced by approximations of the algorithm and the integral is discretized. In order to make use of properties of the least squares projection, we need the analogs of the least squares solutions $\varphi^{q^{N,M}}$ and $\varphi^{z^{N,M}}$ based on the functions $\Xi^N$. Since those  depend on the whole path of $X$ rather than just the values on the time grid $\pi$ additional fictitious simulations are required for the theoretical error analysis which motivates the following definition:
\begin{defi}\label{Simulations}
For $i\in \{0,\ldots, N-1\}$,  let $\mathcal{S}_i:=\{\Delta W^{[i,m]}_{i+1},X^{[i,m]}: m=1,\ldots, M\}$ be a cloud of independent random variables defined on the probability space $(\Omega^M,\mathcal{F}^M,P^M)$ with $X^{[i,m]}=(X^{[i,m]}_s)_{t_i\le s\le T}$  such that $X^{[i,m]}$ is distributed like a segment of the SDE solution $X$. Furthermore, we assume that these simulations match the ones used in the SDP algorithm on the time grid $\pi$, i.e., $\Delta W_{i+1}^{[i,m]}=\Delta W_{i+1}^{[i,m,N]}$ and $X^{[i,m]}_{t_j}=X^{[i,m,N]}_{t_j}$  for all $i\in\{0,\ldots, N-1\}$, $j\ge i$ and $m\in \{1,\ldots,M_i\}$. Then, for every $\omega\in \Omega^M$, let $\nu^M_i(\omega,.)$ be the measure on $\left((\mathbb{R}^{D})^{[t_i,T]}\times \mathbb{R}^\mathcal{D}, \mathcal{B}\left((\mathbb{R}^{D})^{[t_i,T]}\times \mathbb{R}^\mathcal{D}\right)\right)$ defined by 
\begin{align}
\nu^M_i(\omega,B):=\frac{1}{M_i}\sum_{m=1}^{M_i}\delta_{\left(\Delta W^{[i,m]}_{i+1}(\omega),X^{[i,m]}(\omega)\right)}(B)
\end{align}
where $\delta_c(.)$ is the Dirac-measure on $c$.
\end{defi}
\noindent
The following calculations are done on the probability space $(\Omega^M, \mathcal{F}^M,P^M)$, where we suppose that there exists a $\mathcal{D}$-dimensional Brownian motion $W$ on $(\Omega^M, \mathcal{F}^M,P^M)$, which is independent of all simulations and hence also a copy $X$ of the SDE solution independent of the simulations. \\
Given these additional random variables, we denote with $\varphi^{\overline{q}^N}_i$ and $\varphi^{\overline{z}^N}_i$ the solutions of the least-squares problems 
\begin{align}
\varphi^{\overline{q}^N}_i:=\argmin_{\psi\in \mathcal{K}_{q,i}}\left(\frac{1}{M}\sum_{m=1}^{M}\left|\psi\left(X_{t_i}^{[i,m]}\right)-\Xi^N_{i}\left(X^{[i,m]}\right)\right|^2\right)
\end{align} 
and
\begin{align}
\varphi^{\overline{z}^N}_i:=\argmin_{\psi \in \mathcal{K}_{z,i}}\left(\frac{1}{M}\sum_{m=1}^{M_i}\left|\psi\left(X_{t_i}^{[i,m]}\right)-\frac{\Delta W_{i+1}^{[i,m]}}{\Delta}\Xi^N_{i}\left(X^{[i,m]}\right)\right|^2\right).
\end{align}
\begin{remark}
Recall the functions $\varphi_i^{q^{N,M}}$ and $\varphi_i^{z^{N,M}}$ from the SDP algorithm. Since it holds by definition of the ghost sample that $X^{[i,m]}_\pi:=(X^{[i,m]}_{t_j})_{t_i< t_j\in \pi}=X^{[i,m,N]}$, we have 
\begin{align}
 \varphi_i^{q^{N,M}}=\argmin_{\psi\in \mathcal{K}_{q,i}}\left(\frac{1}{M}\sum_{m=1}^{M}\left|\psi\left(X^{[i,m]}_{t_i}\right)-\Xi^{N,M}_{i}\left(X^{[i,m]}_\pi\right)\right|^2\right)
\end{align}
and
\begin{align}
\varphi_i^{z^{N,M}}=\argmin_{\psi \in \mathcal{K}_{z,i}}\left(\frac{1}{M}\sum_{m=1}^{M_i}\left|\psi\left(X^{[i,m]}_{t_i}\right)-\frac{\Delta W_{i+1}^{[i,m]}}{\Delta}\Xi^{N,M}_{i}\left(X^{[i,m]}_\pi\right)\right|^2\right),
\end{align}
i.e., for any fixed outcomes of the simulations $X^{[i,m]}$, both pairs of functions $\varphi_i^{q^{N,M}}$ and $\varphi_i^{z^{N,M}}$, and $\varphi^{\overline{q}^N}_i$ and $\varphi^{\overline{z}^N}_i$ solve a least squares problem with respect to the same measure $\nu_i^{M}(\omega,.)$.
\end{remark}
\noindent
This allows us to utilize the following lemma, which is a key tool in the error analysis. It matches essentially Proposition 4.12 in \cite{gobet2016} where the domain of the function $\Xi$ is generalized in order to cover our setting. The proof presented in \cite{gobet2016} still holds for this setting.  
\begin{lem}\label{lem:projection}
For each $\omega\in \Omega^M$, let $(A,\mathcal{A},\nu(\omega,.))$ be a measurable space with  
\begin{align}
\nu(\omega,.)=\frac{1}{M}\sum_{m=1}^{M}\delta_{\chi^{[m]}(\omega)}
\end{align}
for $i.i.d$ random variables $\chi^{[1]}\ldots, \chi^{[M]}:\Omega^M\rightarrow A$. Furthermore, let $\mathcal{K}$ be a linear function space spanned by $\mathbb{R}^l$-valued basis functions $\{p^{k}(.), 1\le k\le K\}$ with $\sum_{k=1}^{K}E\left[|p^k(\chi^{[m]})|^2\right]<\infty$ for all $m$. For any  $\mathcal{F}^M \otimes\mathcal{A}$-measurable, $\mathbb{R}^l$-valued random variable $\Xi$ with $\Xi(\omega,.)\in L^2\left(\mathcal{A},\nu(\omega,.)\right)$ for $P^M$-a.e. $\omega$, set
\begin{align}
\varphi(\omega,.):= \arginf _{\psi\in \mathcal{K}}\frac{1}{M}\sum_{m=1}^{M}\left|\psi\left(\chi^{[m]}(\omega)\right)-\Xi\left(\omega,\chi^{[m]}(\omega)\right)\right|^2.
\end{align}
Then:
\begin{itemize}
\item[(i)] The mapping $\Xi\mapsto \varphi$ is linear.
\item[(ii)] It holds 
\begin{align}
\|\varphi\|_{L^2\left(\mathcal{A},\nu(\omega,.)\right)}\le \|\Xi\|_{L^2\left(\mathcal{A},\nu(\omega,.)\right)}
\end{align}
where we denote with $\|.\|_{L^2\left((\mathcal{A},\nu(\omega,.)\right)}$ the $L^2$-norm with respect to the measure $\nu(\omega,.)$.
\item[(iii)] Suppose $\mathcal{G}$   is a sub-$\sigma$-field of $\F^M$ such that $\left(p^k(\chi^{[1]}),\ldots, p^k(\chi^{[M]})\right)$ is $\mathcal{G}$-measurable for each $k=1,\ldots, K$. Then 
\begin{align}
E\left[\varphi\middle|\mathcal{G}\right](\omega,.)=\argmin_{\psi\in \mathcal{K}}\frac{1}{M}\sum_{i=1}^{M}\left|\psi\left(\chi^{[m]}(\omega)\right)-\Xi_{\mathcal{G}}\left(\chi^{[m]}(\omega)\right)\right|^2
\end{align}
where $\Xi_{\mathcal{G}}(x):=E\left[\Xi(x)\middle|\mathcal{G}\right]$.
\item[(iv)] 
In the situation of $(iii)$, suppose that $\mathcal{G}$ is given by $\sigma(g(\chi^{[m]})_{m=1,\ldots, M})$ for a $\mathcal{A}$-measurable function 
\begin{align}
g:A&\rightarrow \mathbb{R}^{l'}.
\end{align}
Furthermore, assume that there is a sub-$\sigma$-field $\mathcal{H}$ independent of $\sigma((\chi^{[m]})_{m=1\ldots,M})$ such that $\Xi$ is $\mathcal{H}\otimes\mathcal{A}$-measurable and that  the conditional variance 
\begin{align}
E\left[\left|\Xi\left(\chi^{[m]}\right)-E\left[\Xi\left(\chi^{[m]}\right)\middle|\mathcal{G}\vee \mathcal{H}\right]\right|^2\middle|\mathcal{G}\vee\mathcal{H}\right]
\end{align}
is  $P^M$-almost surely uniformly bounded by some constant $\sigma^2$ for all $m\in \{1,\ldots, M\}$. Then 
\begin{align}
E\left[\|\varphi-E\left[\varphi\middle|\mathcal{G}\vee \mathcal{H}\right]\|^2_{L^2\left(\mathcal{A},\nu(\omega,.)\right)}\middle|\mathcal{G}\vee \mathcal{H}\right]\le \sigma^2\frac{K}{M}.
\end{align}
\end{itemize}
\end{lem}
\noindent
While Lemma \ref{lem:projection} and the objects defined so far help us to utilize projection properties, we still need tools to handle the dependency on the different sets of simulations used in the algorithm. For this purpose, we first consider the following norms allowing us to distinguish between the dependence on simulations or the actual law of the true SDE solution $X$ more clearly:
\begin{defi}
\label{defi:norms}
Let $\varphi:\Omega^{M}\times \mathbb{R}^D\rightarrow \mathbb{R}^l$ be $\mathcal{F}^{M}\times\mathcal{B}(\mathbb{R}^D)$-measurable. For each $i=0,\ldots, N-1$, define the random norms $\|.\|_{i,\infty}$ and $\|.\|_{i,M}$ via
\begin{align}
\|\varphi\|_{i,\infty}^2:=\int_{\mathbb{R}^D}|\varphi(x)|^2P_{X_{t_i}}(dx),\quad\quad\|\varphi\|_{i,M}^2:=\frac{1}{M}\sum_{m=1}^{M}\left|\varphi\left(X_{t_i}^{[i,m]}\right)\right|^2
\end{align}
where $P_{X_{t_i}}$ denotes the distribution of $X_{t_i}$.
\end{defi}
\noindent 
Note that we are interested in the error with respect to the law of the SDE solution $X$, i.e., in the difference between the approximations $q^{N,M}_i$ and  $z^{N,M}_i$, and the functions $\overline{q}^{N}_i$ and $\overline{z}^{N}_i$ respectively in the $\|.\|_{i,\infty}$-norm. The following lemma allows us to lead this difference back to the one in the $\|.\|_{i,M}$-norm that appears naturally in the error analysis due to the use of simulations. It is a straightforward adaptation of Proposition 4.10 in \cite{gobet2016} where the analogous result is shown for $\epsilon=1$ instead of $\epsilon \in (0,1]$. 
\begin{lem}
\label{lem:samplechange}
It holds for all $i=0,\ldots, N-1$ and any $\epsilon\in (0,1]$ that
\begin{align}
E\left[\|q_i^{N,M}-\overline{q}^N_i\|^2_{i,\infty}\right]&\le (1+\epsilon)E\left[\|q_i^{N,M}-\overline{q}^N_i\|^2_{i,M}\right]+\frac{C_1 K_{q,i}\log(C_2M_i)}{M_i\epsilon}\\
E\left[\|z_i^{N,M}-\overline{z}^N_i\|^2_{i,\infty}\right]&\le (1+\epsilon)E\left[\|z_i^{N,M}-\overline{z}^{N}_i\|^2_{i,M}\right]+\frac{\mathcal{D}C_1 K_{z,i}\log(C_2M_i)}{\Delta M_i\epsilon}
\end{align}
for positive constants $C_1,C_2$ independent of $\epsilon$, $\Delta$, and $M_i$.
\end{lem}
\noindent 
Finally, the following lemma allows us to reduce the dependency on a sampled path of $X$ to only the value of the sample at one time point $t_i$. A proof can be found in Chapter 5 of \cite{kallenberg}. 
\begin{lem}\label{lem:pathreduction}
Let $\mathcal{G}$ and $\mathcal{H}$ be independent sub-$\sigma$-fields of $\mathcal{F}^M$. For $l\ge1$ let $F:\Omega^M\times \mathbb{R}^D\rightarrow \mathbb{R}^l$ be bounded and $\mathcal{G}\times \mathcal{B}(\mathbb{R}^D)$-measurable and $U:\Omega^M\rightarrow \mathbb{R}^D$ be $\mathcal{H}$-measurable. Then $E[F(U)|\mathcal{H}]=j(U)$ where $j(x)=E[F(x)]$ for all $x\in \mathbb{R}^D$.
\end{lem}
\noindent 
To see how we can utilize this lemma, note that each sample $X^{[i,m]}$ satisfies
\begin{align}
X_{t_i+s}^{[i,m]}=X_{t_i}^{[i,m]}+\int_{t_i}^{t_i+s}b(l,X_l)dl+\int_{t_i}^{t_i+s}\sigma(l,X_l)dW^{[i,m]}_l
\end{align}
for a Brownian motion $W^{[i,m]}$. Substituting the time variable and setting $\tilde{W}_u=W^{[i,m]}_{u+t_i}-W^{[i,m]}_{t_i}$ leads to
\begin{align}
X_{t_i+s}^{[i,m]}=X_{t_i}^{[i,m]}+\int_{0}^{s}b(u+t_i,X_{u+t_i})du+\int_{0}^{t-s}\sigma(u+t_i,X_{u+t_i})d\tilde{W}_u
\end{align}
which shows that the sample $X^{[i,m]}$ is the solution to a forward SDE starting in $t_i$ with initial value $X_{t_i}^{[i,m]}$ with respect to the Brownian motion $\tilde{W}_u$. Hence, we can express the path $X^{[i,m]}$ as a deterministic function $h$ of  $X^{[i,m]}_{t_i}$ and $(\tilde{W}_{u})_{t_i\le u\le T}$ only, i.e., we can write $X^{[i,m]}_{s}=h(s,X_{t_i}^{[i,m]},(\tilde{W}_u)_{0\le u\le s-t_i})$ for any $s\ge t_i$.  Since the Brownian motion $\tilde{W}$ is independent of $\sigma(X^{[i,m]}_{t_i})$, we can then  use Lemma \ref{lem:pathreduction} on the function
\begin{align}
F(x_{t_i})&=\int_{t_{i+1}}^{t_{\tau(i)+1}}f\left(s,h(s,x_{t_i},(\tilde{W}_u)_{0\le u\le s-t_i}),Y_s,Z_s\right)ds\\
&\quad-\Delta\sum_{j={i+1}}^{\tau(i)}f\left(t_j,h(t_j,x_{t_i},(\tilde{W}_u)_{0\le u\le t_j-t_i}),q^{N,M}_j\left(h(t_j,x_{t_i},(\tilde{W}_u)_{0\le u\le t_j-t_i})\right),\right.\\
&\left.\quad z^{N,M}_j\left(h(t_j,x_{t_i},(\tilde{W}_u)_{0\le u\le t_j-t_i}))\right)\right)
\end{align}
and get
\begin{eqnarray} \label{eq:q1}
&& E\left[\Xi^N_i\left(X^{[i,m]}\right)-\Xi^{N,M}_i\left(X^{[i,m]}\right)\middle|\sigma(\mathcal{S}_j:j>i)\vee \sigma\left(X^{[i,m]}_{t_i}\right)\right]\nonumber\\
&=&E\left[\Xi^N_i\left(\underline{X}_{t_i}\right)-\Xi^{N,M}_i\left(\underline{X}_{t_i}\right)\middle|\sigma(\mathcal{S}_j:j>i), X_{t_i}=X^{[i,m]}_{t_i}\right]
\end{eqnarray}
where we set $\underline{X}_{t_i}=(X_s)_{s\in [t_i,T]}$. Completely analogously, it holds
\begin{eqnarray} \label{eq:z1}
&&E\left[\frac{\Delta W^{[i,m]}_{i+1}}{\Delta}\left(\Xi^{N}_i(X^{[i,m]})-\Xi^{N,M}_i(X^{[i,m]})\right)\middle|\sigma(\mathcal{S}_j:j>i)\vee \sigma\left(X^{[i,m]}_{t_i}\right)\right]\nonumber\\
&=&E\left[\frac{\Delta W^{[i,m]}_{i+1}}{\Delta}\left(\Xi^{N}_i(\underline{X}_{t_i})-\Xi^{N,M}_i(\underline{X}_{t_i})\right)\middle|\sigma(\mathcal{S}_j:j>i), X_{t_i}=X^{[i,m]}_{t_i}\right].
\end{eqnarray}
Hence, Lemma \ref{lem:pathreduction} allows us to reduce the dependency on a sample path to a conditional expectation according to the actual law of $X$ given the value of the sample path at the current time.\\
We close this section with some abbreviations for the notation:
Throughout the rest of this chapter, we denote with $\mathcal{F}^M_i:=\sigma(\mathcal{S}_k:~k> i)\vee \sigma(X^{[i,m]}_{t_i}: m=1,\ldots, M_i)$  the $\sigma$-field generated by the simulations used up to the time $k$ (backwards starting from $N$) for a fixed $N\in \mathbb{N}$. With $\mathcal{F}_i:=\sigma((W_s)_{0\le s\le t_i})$ we denote the $\sigma$-field generated by the Brownian motion which is independent of the simulations.  The conditional expectations given those $\sigma$-fields we denote with $E_i^M[.]=E[.|\mathcal{F}^M_i]$ and $E_i[.]:=E[.|\mathcal{F}_i]$.  Additionally, we shorten the notation of the driver $f$ by dropping clearly indicated function arguments through the notation $f(t,x,g,g'):=f(t,x,g(x),g'(x))$ for any  functions $g,g':\mathbb{R}^D\rightarrow \mathbb{R}^l$ and $f(t,x,g_t,g'_t):=f(t,x,(g(t,x),g'(t,x))$ for any functions $g,g':[0,T]\times \mathbb{R}^D\rightarrow \mathbb{R}^l$.\\
After these considerations we are ready for the derivation of the error bounds.
\subsection{Error of the Q approximation}
In this section, we analyze the expected quadratic error of the approximation of $\overline{q}^N$ in the form of the terms 
\begin{align}
E\left[\|q^{N,M}_i-\overline{q}^N_i\|^2_{i,M}\right].
\end{align}
We first establish a bound on the error propagation between the time steps, allowing us to express the expected difference of this term by the one in the next time step. Afterwards, we derive a local error bound for the term as well as a bound for the global error by bounding the maximum of the quadratic error terms  over all time points $t_i\in \pi$. The first part of the proof, in which we establish the error propagation, is  inspired by the error analysis of the MDP scheme in \cite{gobet2016}.\\
\emph{Error propagation:}
First, note that 
\begin{align}
E^M_i\left[\Xi^N_i\left(X^{[i,m]}\right)\right]&=E^M_i\left[\overline{q}^N_{\tau_\alpha(i)}\left(X^{[i,m]}_{t_{\tau_\alpha(i)}}\right)+\int_{t_{i+1}}^{t_{\tau_\alpha(i)+1}}f\left(s,X^{[i,m]}_s,\overline{y}_s,\overline{z}_s\right)ds\right]\\
&=E\left[\overline{q}^N_{\tau_\alpha(i)}\left(X_{t_{\tau_\alpha(i)}}\right)+\int_{t_{i+1}}^{t_{\tau_\alpha(i)+1}}f\left(s,X_s,Y_s,Z_s\right)ds\middle|X_{t_i}=X^{[i,m]}_{t_i}\right]\\
&=\overline{q}^N_i\left(X^{[i,m]}_{t_i}\right).\label{eq:bestapprox}
\end{align}
Hence, it holds by Lemma \ref{lem:projection} (iii) that $E^M_i[\varphi^{\overline{q}^N}_i]$ is the least squares projection of $\overline{q}^N_i$ on the subspace $\mathcal{K}_{q,i}$ with respect to the measure $\nu^M_i$, i.e., it holds 
\begin{align}
E^M_i\left[\varphi^{\overline{q}^N}_i\right]=\arginf_{\psi\in \mathcal{K}_{q,i}}\left(\frac{1}{M}\sum_{i=1}^M\left|\psi\left(X_{t_i}^{[i,m]}\right)-\overline{q}^N_i\left(X_{t_i}^{[i,m]}\right)\right|^2\right).
\end{align}
Therefore, by the properties of least square projections, $\overline{q}^N_i-E^M_i[\varphi^{\overline{q}^N}_{i}]$ is orthogonal on $E_i^M[\varphi^{\overline{q}^N}_{i}]-\psi^{q^{N,M}}_{i}$ with respect to $\nu^M_i$. Additionally, note that, since $\xi$ and $f$ are bounded due to the assumptions $(A_\xi)$ and $(A_f)$, it follows by a simple backward recursion that 
\begin{align}
|\overline{q}^N_i(x)|\le C_{q,i}=C_\xi+(T-t_{i+1})C_f\le C_q:=C_\xi+TC_f
\end{align}
for all $x\in \R^D$. We conclude that $\mathcal{T}_{C_{q,i}}(\overline{q}^N_i(x))=\overline{q}^N_i(x)$ and obtain:
\begin{align}
E\left[\|\overline{q}^N_i(.)-q_i^{N,M}(.)\|^2_{i,M}\right]	&=E\left[\left\|\mathcal{T}_{C_{q,i}}\left(\overline{q}^N_i(.)\right)-\mathcal{T}_{C_{q,i}}\left(\varphi^{q^{N,M}}_i(.)\right)\right\|^2_{i,M}\right]\\
&\le E\left[\left\|\overline{q}^N_i(.)-E^M_i\left[\varphi^{\overline{q}^N}_{i}(.)\right]+E_i^M\left[\varphi^{\overline{q}^N}_{i}(.)\right]-\varphi^{q^{N,M}}_{i}(.)\right\|^2_{i,M}\right]\\
&= E\left[\left\|\overline{q}^N_i(.)-E^M_i\left[\varphi^{\overline{q}^N}_{i}(.)\right]\right\|_{i,M}^2\right]+E\left[\left\|E_i^M\left[\varphi^{\overline{q}^N}_{i}(.)\right]-\varphi^{q^{N,M}}_{i}(.)\right\|^2_{i,M}\right]\\
&\le E\left[\left\|\overline{q}^N_i(.)-E_i^M\left[\varphi^{\overline{q}^N}_{i}(.)\right]\right\|^2_{i,M}\right]\\
&\quad+(1+\kappa)E\left[\left\|E_i^M\left[\varphi^{q^{N,M}}_{i}(.)-\varphi^{\overline{q}^N}_{i}(.)\right]\right\|_{i,M}^2\right]\\
&\quad +(1+\kappa^{-1})E\left[\left\|E_i^M\left[\varphi^{q^{N,M}}_{i}(.)\right]-\varphi^{q^{N,M}}_{i}(.)\right\|^2_{i,M}\right]
\end{align}
where  $\kappa$ is an arbitrary positive constant. We handle the terms separately:\\
As argued before, it holds
\begin{align}
E_i^M\left[\varphi^{\overline{q}^N}_{i}(.)\right]&=\arginf_{\psi\in \mathcal{K}_{q,i}} \|\overline{q}^N_i(.)-\psi(.)\|_{i,M}^2.
\end{align}
Hence,
\begin{align}
E\left[\left\|\overline{q}^N_i(.)-E_i^M\left[\varphi^{\overline{q}^N}_{i}(.)\right]\right\|^2_{i,M}\right]&=E\left[\inf_{\psi\in \mathcal{K}_{q,i}} \big\|\overline{q}^N_i(.)-\psi(.)\big\|_{i,M}^2\right]\\
&\le \inf_{\psi\in \mathcal{K}_{q,i}} \frac{1}{M}\sum_{i=1}^M E\left[\big|\overline{q}^N_i\left(X^{[i,m]}_{t_i}\right)-\psi\left(X^{[i,m]}_{t_i}\right)\big|^2\right]\\
&=\inf_{\psi\in \mathcal{K}_{q,i}} E\left[\left|\overline{q}^N_i(X_{t_i})-\psi(X_{t_i})\right|^2\right].
\end{align}
This term describes the best approximation error  due to the projection on the subspace $\mathcal{K}_{q,i}$ and is part of the final error representation.\\
For the next term, note that $\Xi^{N,M}_{i}$ is bounded by $C_{q,i}$ under assumptions $(A_\xi)$ and $(A_f)$, since the approximations $q_i^{N,M}$ are (due to the truncation in the algorithm).  Additionally, for each $i$, the function $\Xi^{N,M}_i$ is built using only the simulations in the sets $\mathcal{S}_k$ for $k>i$. Hence, it follows directly by Lemma \ref{lem:projection} (iv) with $\mathcal{H}=\sigma(\mathcal{S}_k, k> i)$ and $g(X^{[i,m]})=X_{t_i}^{[i,m]}$ that 
\begin{align}
E\left[\left\|E_i^M\left[\varphi^{q^{N,M}}_{i}(.)\right]-\varphi^{q^{N,M}}_{i}(.)\right\|^2_{i,M}\right]\le  C^2_{q,i}\frac{K_{q,i}}{M_i}.
\end{align}
Since, as argued before,  the functions $\overline{q}^N_i$ are bounded by $C_{q,i}$ as well, we can apply a similar argument  to the functions $\Xi^N_i$. Those are  again built using only the simulations in $\mathcal{S}_k$ for $k>i$. Setting $\xi^q_{i}(x):=E[\Xi^{N,M}_{i}(\underline{X}_{t_i})-\Xi^N_{i}(\underline{X}_{t_i})|X_{t_i}=x, \F_0^M]$, we have by \eqref{eq:q1} that 
\begin{align}
E_i^M\left[\Xi^{N,M}_{i}(X^{[i,m]})-\Xi^N_{i}(X^{[i,m]})\right]=\xi^q_{i}(X_i^{[i,m]}).
\end{align}
Then Lemma \ref{lem:projection} (i) and (iii) imply that 
\begin{align}
E_i^M\left[\varphi^{q^{N,M}}_{i}(.)-\varphi^{\overline{q}^N}_{i}(.)\right]=\arginf_{\psi\in \mathcal{K}_{q,i}}\left(\frac{1}{M}\sum_{i=1}^M\Big|\psi\left(X_{t_i}^{[i,m]}\right)-\xi^q_i\left(X_{t_i}^{[i,m]}\right)\Big|^2\right).
\end{align}
Hence, by Lemma \ref{lem:projection} (ii) it holds 
\begin{align}
E\left[\big\|E_i^M\left[\varphi^{q^{N,M}}_{i}(.)-\varphi^{\overline{q}^N}_{i}(.)\right]\big\|_{i,M}^2\right]\le E\left[\big\|\xi^q_i(.)\big\|_{i,M}^2\right]=E\left[\xi^q_i\left(X_{t_i}\right)^2\right].
\end{align}
Plugging in the estimates derived so far we have:
\begin{eqnarray}\label{eq:q2}
\begin{split}
E\left[\big\|\overline{q}^N_i(.)-q_i^{N,M}(.)\big\|^2_{i,M}\right]&\le  \inf_{\psi\in \mathcal{K}_{q,i}} E\left[\left(\overline{q}^N_i(X_{t_i})-\psi(X_{t_i})\right)^2\right]\\
&\quad+(1+\kappa^{-1})C^2_{q,i}\frac{K_{q,i}}{M_i} +(1+\kappa)E\left[\xi^q_i(X_{t_i})^2\right].
\end{split}
\end{eqnarray}
To further estimate $E[\xi^q_i(X_{t_i})^2]$,  we have to distinguish between the time points. If the time points $t_{i}$ and $t_{i+1}$ are in the same  segment  defined by $\overline{\pi}_\alpha$, i.e., at time points $t_i$ such that $t_{i+1}\not \in \overline{\pi}_\alpha$, it holds  $\tau_\alpha(i)=\tau_\alpha(i+1)$ by our choice of $\tau_\alpha$ and hence:
\begin{align}
\Xi^{N,M}_i\left(\underline{X}_{t_i}\right)&=q^{N,M}_{\tau_\alpha(i)}\left(X_{t_{\tau_\alpha(i)}}\right)+\sum_{j=i+1}^{\tau_\alpha(i)}\Delta f\left(t_j,X_{t_j},q_j^{N,M},z_j^{N,M}\right)\\
&=\Xi^{N,M}_{i+1}\left(\underline{X}_{t_{i+1}}\right)+\Delta f\left(t_{i+1},X_{t_{i+1}},q_{i+1}^{N,M},z_{i+1}^{N,M}\right)
\end{align}
and
\begin{align}
\Xi^{N}_i\left(\underline{X}_{t_i}\right)&=\overline{q}^N_{\tau_\alpha(i)}\left(X_{t_{\tau_\alpha(i)}}\right)+\int_{t_{i+1}}^{t_{\tau_\alpha(i)+1}}f\left(t,X_{t},Y_t,Z_t\right)dt\\
&=\Xi^{N}_{i+1}\left(\underline{X}_{t_{i+1}}\right)+\int_{t_{i+1}}^{t_{i+2}}f\left(t,X_{t},Y_t,Z_t\right)dt.
\end{align}
This allows us to estimate  $E[\xi^q_i(X_{t_i})^2]$ as:
\begin{align}
&E\left[\xi^{q}_i(X_{t_i})^2\right]\\
&=E\bigg[E\Big[\overline{q}^N_{\tau_\alpha(i)}\left(X_{t_{\tau_\alpha(i)}}\right)-q^{N,M}_{\tau_\alpha(i)}\left(X_{t_{\tau_\alpha(i)}}\right)\\
&\quad +\sum_{j=i+1}^{\tau_\alpha(i)}\int_{t_j}^{t_{j+1}}f\left(s,X_s,Y_s,Z_s\right)-f\left(t_j,X_{t_j},q_j^{N,M},z_j^{N,M}\right)ds\Big|\mathcal{F}_0^M,X_{t_i}\Big]^2\bigg]\\
&\le (1+\Gamma\Delta)E\bigg[E\Big[E\Big[\Xi^N_{i+1}(\underline{X}_{t_{i+1}})-\Xi^{N,M}_{i+1}(\underline{X}_{t_{i+1}})    \Big| \mathcal{F}^M_0, X_{t_i}, X_{t_{i+1}}\Big]\Big|\mathcal{F}_0^M,X_{t_i}\Big]^2\bigg]\\
&\quad + (1+\frac{1}{\Gamma\Delta})E\bigg[E\Big[\int_{t_{i+1}}^{t_{i+2}}f\left(s,X_s,Y_s,Z_s\right)-f\left(t_{i+1},X_{t_{i+1}},q_{i+1}^{N,M},z_{i+1}^{N,M}\right)ds\Big|\mathcal{F}_0^M,X_{t_i}\Big]^2\bigg]\\
&\le  (1+\Gamma\Delta)E\left[\left(\xi^q_{i+1}(X_{t_{i+1}})\right)^2\right]\\
&\quad + (1+\frac{1}{\Gamma\Delta})E\bigg[E\Big[\int_{t_{i+1}}^{t_{i+2}}f\left(s,X_s,Y_s,Z_s\right)-f\left(t_{i+1},X_{t_{i+1}},q_{i+1}^{N,M},z_{i+1}^{N,M}\right)ds\Big|\mathcal{F}_0^M,X_{t_i}\Big]^2\bigg].
\end{align}
Here we first used Young's inequality with some constant $\Gamma>0$ that will be specified later and the tower property of the conditional expectation in the first inequality. Then in the second step, we used Jensen's inequality, the Markov property of $X$ and once more the tower property of the conditional expectation. The calculation shows that the expected error term $E[\xi^{q}_i(X_i)^2]$ is bounded by the one in the next time step and a driver-dependent  term. \\
At time points at the end of a segment, i.e., for time points $t_i$ such that $t_{i+1}\in \overline{\pi}_\alpha$, the function $\Xi^{N,M}_i$ and $\Xi^{N,M}_{i+1}$ are defined on different segments and hence the first inequality in the calculations above does not hold true in this case.  However, to get a similar bound, we can once more use  Young's inequality and a zero addition to get 
\begin{align}
&E\left[\left(\xi^{q}_i(X_{t_i})\right)^2\right]\\
&=E\bigg[E\Big[\overline{q}^N_{\tau_\alpha(i)}\left(X_{t_{\tau_\alpha(i)}}\right)-q_{\tau_\alpha(i)}^{N,M}\left(X_{t_{\tau_\alpha(i)}}\right)\\
&\quad+\int_{t_{i+1}}^{t_{i+2}}f\left(s,X_s,Y_s,Z_s\right)-f\left(t_{i+1},X_{t_{i+1}},q_{i+1}^{N,M},z_{i+1}^{N,M}\right)ds\Big|\mathcal{F}_0^M,X_{t_i}\Big]^2\bigg]\\
&\le (1+\Gamma\Delta) E\Big[\big\|\overline{q}^N_{i+1}-q^{N,M}_{i+1}\big\|^2_{i+1,\infty}\Big]\\
&\quad + (1+\frac{1}{\Gamma\Delta}) E\bigg[E\Big[\int_{t_{i+1}}^{t_{i+2}}f\left(s,X_s,Y_s,Z_s\right)-f\left(t_{i+1},X_{t_{i+1}},q_{i+1}^{N,M},z_{i+1}^{N,M}\right)ds\Big|\mathcal{F}_0^M,X_{t_i}\Big]^2\bigg]\\
& \le (1+\Gamma \Delta)(1+\kappa)(1+\epsilon)E\left[ \left(\xi^{q}_{i+1}(X_{t_{i+1}})\right)^2\right]\\
&\quad +(1+\frac{1}{\Gamma\Delta})E\bigg[E\Big[\int_{t_{i+1}}^{t_{i+2}}f\left(s,X_s,Y_s,Z_s\right)-f\left(t_{i+1},X_{t_{i+1}},q_{i+1}^{N,M},z_{i+1}^{N,M}\right)ds\Big|\mathcal{F}_0^M,X_{t_i}\Big]^2\bigg]\\
&\quad +(1+\Gamma\Delta)\left(E\left[\big\|\overline{q}^N_{i+1}-q^{N,M}_{i+1}\big\|^2_{i+1,\infty}\right]-(1+\kappa)(1+\epsilon)E\left[\left(\xi^{q}_{i+1}(X_{t_{i+1}})\right)^2\right]\right)_+
\end{align}
with positive constants $\kappa$ and $ \epsilon$, which we will specify later. Again we have bounded the error term $E[\xi^q_i(X_{t_i})^2]$ by the one at the next time point and a driver-dependent term, but now with an additional error term that depends on the approximation  of $\overline{q}^N$ at the time point $t_{i+1}$. This additional error term could be expected since the SDP algorithm works on the segment containing $t_i$ like the MDP scheme where the correct terminal condition of the BSDE restricted to the corresponding time segment has been replaced by the approximation  $q^N_{\tau_\alpha(i)}$ at the time point $\tau_\alpha(i)=t_{i+1}$.\\
\emph{Local and global bounds:}\\
Iterating this step leads to the following local and global bounds for the quadratic error of the approximation of $\overline{q}^N$, that are stated in terms of  the expectations $E[\xi^q_i(X_i)^2]$ for later use. To obtain a corresponding bound for the terms $E[\|\overline{q}^{N}_i-q^{N,M}_i\|^2_{i,M}]$, one can simply follow the arguments that  lead to \eqref{eq:q2} and apply the lemma afterwards.
\begin{lem}\label{lem:Q}
For a positive constant $\Gamma$, set $\lambda_i:=(1+\Gamma\Delta)^i(1+N^{\alpha-1})^{2|\{j\le i: t_j\in \overline{\pi}_\alpha\}|}$ for $i\in \{0,\ldots,N\}$. It then holds under the standing assumptions that
\begin{align}
&E\left[\left(\xi_i^q(X_{t_i})\right)^2\right]\\
&\le \lambda_iE\left[\left(\xi_i^q(X_{t_i})\right)^2\right]\\
&\le (1+\frac{1}{\Delta\Gamma})\sum_{j=i}^{N-2}\lambda_{j}E\left[E\left[\int_{t_{j+1}}^{t_{j+2}}f\left(s,X_s,Y_s,Z_s\right)-f\left(t_{j+1},X_{t_{j+1}},q_{j+1}^{N,M},z_{j+1}^{N,M}\right)ds\middle|\F_0^M,X_{t_j}\right]^2\right]\\
&\quad +\lceil N^{1-\alpha}\rceil\lambda_N\sup_{j\in \mathcal{I}}\left(E\left[\big\|\overline{q}^N_{j}-q_{j}^{N,M}\big\|_{j,\infty}^2\right]-(1+N^{\alpha-1})^2E\left[\left(\xi^{q}_{j}(X_{t_{j}})\right)^2\right]\right)_+
\end{align}
for all $i\in \{0,\ldots, N-2\}$ and 
\begin{align}
&\max_{0\le i\le N-1}E\left[\left(\xi_i^q(X_{t_i})\right)^2\right] \\
&\le \sum_{j=0}^{N-2}(1+\frac{1}{\Delta\Gamma})\lambda_jE\left[E\left[\int_{t_{j+1}}^{t_{j+2}}f\left(s,X_s,Y_s,Z_s\right)-f\left(t_{j+1},X_{t_{j+1}},q_{j+1}^{N,M},z_{j+1}^{N,M}\right)ds\middle|\F_0^M,X_{t_j}\right]^2\right]\\
&\quad + \lceil N^{1-\alpha}\rceil\lambda_N\max_{j\in \mathcal{I}}\left((1+N^{\alpha-1})\left(\inf_{\psi\in \mathcal{K}_{q,j}} E\left[\left|\overline{q}^N_j(X_{t_j})-\psi(X_{t_j})\right|^2\right]+(1+N^{1-\alpha})\frac{C^2_{q,j}K_{q,j}}{M_j}\right)\right.\\
&\quad \left.+\frac{N^{1-\alpha}C_1K_{q,j}\log(C_2M_{j})}{M_{j}}\right).
\end{align}
where $\mathcal{I}:=\{i: t_i\in \overline{\pi}_\alpha\}$.
\end{lem}
\begin{proof}
Iterating the previous calculations where we choose $\kappa=\epsilon=N^{\alpha-1}$ yields
\begin{align}
&E\left[\left(\xi_i^q(X_{t_i})\right)^2\right]\le \lambda_iE\left[\left(\xi_i^q(X_{t_i})\right)^2\right] \\
&\le \lambda_{i+1}E\left[\left(\xi^{q}_{i+1}(X_{t_{i+1}})\right)^2\right]\\
&\quad +\lambda_i(1+\frac{1}{\Delta\Gamma})E\left[E\left[\int_{t_{i+1}}^{t_{i+2}}f\left(s,X_s,Y_s,Z_s\right)-f\left(t_{i+1},X_{t_{i+1}},q_{i+1}^{N,M},z_{i+1}^{N,M}\right)ds\middle|\F_0^M,X_{t_i}\right]^2\right]\\
&\quad +\lambda_{i+1}\left(E\left[\big\|\overline{q}^N_{i+1}-q_{i+1}^{N,M}\big\|^2_{i+1,\infty}\right]-(1+N^{\alpha-1})^2E\left[\left(\xi^{q}_{i+1}(X_{t_{i+1}})\right)^2\right]\right)_+ \mathds{1}_{\mathcal{I}}(i+1)\\
&\le \lambda_{N-1} E\left[\left(\xi^{q}_{N-1}(X_{t_{N-1}})\right)^2\right]\\
&\quad +(1+\frac{1}{\Delta\Gamma})\sum_{j=i}^{N-2}\lambda_{j}E\left[E\left[\int_{t_{j+1}}^{t_{j+2}}f\left(s,X_s,Y_s,Z_s\right)-f\left(t_{j+1},X_{t_{j+1}},q_{j+1}^{N,M},z_{j+1}^{N,M}\right)ds\middle|\F_0^M,X_{t_j}\right]^2\right]\\
&\quad +\lceil N^{1-\alpha}\rceil\lambda_N\max_{j\in \mathcal{I}}\left(E\left[\big\|\overline{q}^N_{j}-q_{j}^{N,M}\big\|_{j,\infty}^2\right]-(1+N^{\alpha-1})^2E\left[\left(\xi^{q}_{j}(X_{t_{j}})\right)^2\right]\right)_+.
\end{align}
Then we have by definition $E[\xi^{q}_{N-1}(X_{t_{N-1}})^2]=0$ since $\Xi^{N,M}_{N-1}=\Xi^N_{N-1}$ and the recursion terminates. This already proves the first statement of Lemma \ref{lem:Q}. Additionally, using Lemma \ref{lem:samplechange} with $\epsilon=N^{\alpha-1}$ and  the inequality \eqref{eq:q2} for $\kappa=N^{\alpha-1}$ we have for any $j$:
\begin{eqnarray}\label{eq:q3}
\begin{split}
&\left(E\left[\big\|\overline{q}^N_{j}-q_j^{N,M}\big\|^2_{j,\infty}\right]-(1+N^{\alpha-1})^2E\left[\left(\xi^{q}_{j}(X_{t_j})\right)^2\right]\right)_+\\
&\le\Bigg((1+N^{\alpha-1})E\left[\big\|\overline{q}^N_{j}-q_j^{N,M}\big\|^2_{j,M}\right]+\frac{N^{1-\alpha}C_1K_{q,j}\log(C_2M_{j})}{ M_{j}}\\
&\quad -(1+N^{\alpha-1})^2E\left[\left(\xi^{q}_{j}(X_{t_j})\right)^2\right]\Bigg)_+\\
&\le\left((1+N^{\alpha-1})\left(\inf_{\psi\in \mathcal{K}_{q,j}} E\left[\left|\overline{q}^N_j(X_{t_j})-\psi(X_{t_j})\right|^2\right]+(1+N^{1-\alpha})\frac{C^2_{q,j}K_{q,j}}{M_j}\right.\right.\\
&\quad \left.\left.+(1+N^{\alpha-1})E\left[\left(\xi^{q}_{j}(X_{t_j})\right)^2\right]\right)+\frac{N^{1-\alpha}C_1K_{q,j}\log(C_2M_{j})}{\epsilon M_{j}}-(1+N^{\alpha-1})^2E\left[\left(\xi^{q}_{j}(X_{t_j})\right)^2\right]\right)_+\\
&\le(1+N^{\alpha-1})\left(\inf_{\psi\in \mathcal{K}_{q,j}} E\left[\left|\overline{q}^N_j(X_{t_j})-\psi(X_{t_j})\right|^2\right]+(1+N^{1-\alpha})\frac{C^2_{q,j}K_{q,j}}{M_j}\right)+\frac{N^{1-\alpha}C_1K_{q,j}\log(C_2M_{j})}{M_{j}}. 
\end{split}
\end{eqnarray}
The second statement of Lemma \ref{lem:Q} then follows by plugging in the estimate above in the first statement and taking the maximum over all time points.
\end{proof}
\subsection{Error of the Z approximation}
Analogously to the previous section, we now analyze the quadratic error of the approximation of $\overline{z}^N$ via the terms $E[\|\overline{z}^N_i-z_i^{N,M}\|^2_{i,M}]$. Again, we first establish a bound on the error propagation between the time steps before deriving global bounds.\\
\emph{Error propagation:} While the later steps require changes, we can get an analog of the inequality \eqref{eq:q2} by applying the same arguments as before. It holds that
\begin{align}
E_i^M\left[\frac{\Delta W^{[i,m]}_{i+1}}{\Delta} \Xi^N_i\left(X^{[i,m]}\right)\right]&=\overline{z}^N_i\left(X_{t_i}^{[i,m]}\right)
\end{align}
and therefore, we have by  by Lemma \ref{lem:projection} (iii)
\begin{eqnarray}\label{eq:z2}
E^M_i\left[\varphi^{\overline{z}^N}_i\right]=\arginf_{\psi\in \mathcal{K}_{z,i}}\left(\frac{1}{M}\sum_{i=1}^M\Big|\psi\left(X_{t_i}^{[i,m]}\right)-\frac{\Delta W^{[i,m]}_{i+1}}{\Delta}\overline{z}^N_i\left(X_{t_i}^{[i,m]}\right)\Big|^2\right).
\end{eqnarray}
We conclude that  $\overline{z}^N _i-E^M_i[\varphi^{\overline{z}^N}_{i}]$ is orthogonal on $E_i^M[\varphi^{\overline{z}^N}_{i}]-\varphi^{z^{N,M}}_{i}$ with respect to $\|.\|_{i,M}$. Additionally,  since $|\Xi^N_i|\le C_{q,i}$ due to assumptions $(A_\xi)$ and $(A_f)$, it holds 
\begin{align}
\big|\overline{z}^{N,(d)}_i(x)\big|\le C_{z,i}=\frac{C_{q,i}}{\Delta}
\end{align}
for each component $\overline{z}_i^{N,(d)}$, $d=1,\ldots, \mathcal{D}$ of $\overline{z}_i^N$, $x\in \mathbb{R}^D$ and  $i\in \{0,\ldots, N-1\}$. With that, we obtain for an arbitrary $\kappa>0$ that
\begin{align}
\Delta E\left[\big\|\overline{z}^N_i-z_i^{N,M}\big\|^2_{i,M}\right]&=\Delta E\left[\big\|\mathcal{T}_{C_{z,i}}\left(\overline{z}^N_i(.)\right)-\mathcal{T}_{C_{z,i}}\left(\varphi_i^{z^{N,M}}(.)\right)\big\|_{i,M}^2\right]\\
&\le \Delta  E\left[\big\|\overline{z}_i^{N}(.)-E_i^M\left[\varphi^{\overline{z}^N}_i(.)\right]+E_i^M\left[\varphi^{\overline{z}^N}_i(.)\right]-\varphi^{z^{N,M}}_i(.)\big\|^2_{i,M}\right]\\
&= \Delta \left(E\left[\big\|\overline{z}_i^N(.)-E_i^M\left[\varphi^{\overline{z}^N}_i(.)\right]\big\|^2_{i,M}\right]+E\left[\big\|E_i^M\left[\varphi^{\overline{z}^N}_i(.)\right]-\varphi^{z^{N,M}}_i(.)\big\|^2_{i,M}\right]\right)\\
&\le \Delta\left(E\left[\big\|\overline{z}^N_i(.)-E^M_i\left[\varphi^{\overline{z}^N}_{i}(.)\right]\big\|^2_{i,M}\right]\right.\\
&\quad +(1+\kappa^{-1})+E\left[\big\|E^M_i\left[\varphi^{z^{N,M}}_{i}(.)\right]-\varphi^{z^{N,M}}_{i}(.)\big\|_{i,M}^2\right]\\
&\quad \left.+(1+\kappa)E\left[\big\|E_i^M\left[\varphi^{z^{N,M}}_{i}(.)-\varphi^{\overline{z}^N}_{i}(.)\right]\big\|_{i,M}^2\right]\right).
\end{align}
Once more we handle the appearing terms separately:\\
First, analogously to the corresponding term in  the previous section, it follows due to equation \eqref{eq:z2} that  
\begin{align}
E[\|\overline{z}^N_i(.)-E^M_i[\varphi^{\overline{z}^N}_{i}(.)]\|^2_{i,M}]\le \inf_{\psi \in \mathcal{K}_{z,i}}E[|\psi(X_{t_i})-\overline{z}^N_i(X_{t_i})|^2],
\end{align}
which describes the best approximation error of $\overline{z}^N$ using the basis functions and is part of the final error representation.\\
For the next term, note again that  $\Xi^{N,M}_i$ is bounded by $C_{q,i}$ for all $i\in \{0,\ldots, N-1\}$. We conclude that
\begin{align}
E^M_i\left[\Bigg|\frac{\Delta W^{[i,m]}_{i+1}}{\Delta}\Xi^{N,M}_i(X^{[i,m]})-E^M_i\left[\frac{\Delta W^{[i,m]}_{i+1}}{\Delta}\Xi^{N,M}_i(X^{[i,m]})\right]\Bigg|^2\right]&\le  E^M_i\left[\Big|\frac{\Delta W^{[i,m]}_{i+1}}{\Delta}\Xi^{N,M}_i(X^{[i,m]})\Big|^2\right]\\
&\le \frac{\mathcal{D}C^2_{q,i}}{\Delta}.
\end{align}
Then, since $\Xi^{N,M}_i$ is built using only simulations of the clouds $\mathcal{S}_k$ for $k>i,$ it follows by Lemma \ref{lem:projection} (iv) that  $E[\|E_i^M[\varphi^{z^{N,M}}_{i}]-\varphi^{z^{N,M}}_{i}\|^2_{i,M}]$ is bounded by $C_{q,i}^2\frac{\mathcal{D}K_{z,i}}{\Delta M_i}$.\\
%Since $\Xi^{N}_i$ is build using only simulations in $\mathcal{S}_k$ for $k>i$ as well, we get with with  Lemma 4.6 that 
For the last term, we have by \eqref{eq:z1} that
\begin{align}
E_i^M\left[\frac{\Delta W^{[i,m]}_{i+1}}{\Delta}\left(\Xi_i^{N,M}(X^{[i,m]})-\Xi^{N}_i(X^{[i,m]})\right)\right]=\xi^{z}_i(X^{[i,m]}_{t_i})
\end{align}
with 
\begin{align}
\xi^{z}_i(x)=E\left[\frac{\Delta W_{i+1}}{\Delta}\left(\Xi^{N}_{i}(\underline{X}_{t_i})-\Xi^{N,M}_{i}(\underline{X}_{t_i})\right)\middle|X_{t_i}=x,\mathcal{F}_0^M\right].
\end{align}
Then, by Lemma \ref{lem:projection} (i) and (iii), it follows that 
\begin{align}
E_i^M\left[\varphi^{z^{N,M}}_{i}(.)-\varphi^{\overline{z}^N}_{i}(.)\right]=\arginf_{\psi\in \mathcal{K}_{z,i}}\left(\frac{1}{M}\sum_{i=1}^M\Big|\psi\left(X_{t_i}^{[i,m]}\right)-\xi_i^z\left(X_{t_i}^{[i,m]}\right)\Big|^2\right).
\end{align}
Hence, by Lemma \ref{lem:projection} (ii) we have
\begin{align}
E\left[\big\|E_i^M\left[\psi^{z^{N,M}}_{i}-\psi^{\overline{z}^N}_{i}\right]\big\|_{i,M}^2\right]\le E\left[\big\|\xi^{z}_i\big\|_{i,M}^2\right]=E\left[\left(\xi^{z}_i(X_{t_i})\right)^2\right].
\end{align}
Plugging in the estimates obtained so far we have
\begin{eqnarray}\label{eq:z3}
\begin{split}
\Delta E\left[\big\|\overline{z}^N_i-z_i^{N,M}\big\|^2_{i,M}\right]&\le \Delta \left(\inf_{\psi \in \mathcal{K}_{z,i}}E\left[\big|\psi\left(X_{t_i}\right)-\overline{z}^N_i\left(X_{t_i}\right)\big|^2\right]\right.\\
&\quad\left.+(1+\kappa^{-1})C_{q,i}^2\frac{\mathcal{D}K_{z,i}}{\Delta M_i}+(1+\kappa)E\left[\left(\xi^{z}_i(X_{t_i})\right)^2\right]\right).
\end{split}
\end{eqnarray}
Now using the tower property and a zero addition, we get
\begin{align}
\Delta E\left[\left(\xi^{z}_i(X_{t_i})\right)^2\right]&=\Delta E\left[E\left[\frac{\Delta W_{i+1}}{\Delta}\left(\Xi^{N,M}_{i}(\underline{X}_{t_i})-\Xi^N_{i}(\underline{X}_{t_i})\right)\middle|\mathcal{F}_0^{M}, X_{t_i}\right]^2\right]\\
&=\Delta E\left[E\left[\frac{\Delta W_{i+1}}{\Delta}E\left[\Xi^{N,M}_{i}(\underline{X}_{t_i})-\Xi^N_{i}(\underline{X}_{t_i})\middle|\mathcal{F}_0^M, \mathcal{F}_{t_{i+1}}\right]\middle|\mathcal{F}_0^M, X_{t_i}\right]^2\right]\\
&=\Delta E\left[E\left[\frac{\Delta W_{i+1}}{\Delta}\left(E\left[\Xi^{N,M}_{i}(\underline{X}_{t_i})-\Xi^N_{i}(\underline{X}_{t_i})\middle|\mathcal{F}_0^M, \mathcal{F}_{t_{i+1}}\right]\right.\right.\right.\\
&\quad \left.\left.\left.-E\left[\Xi^{N,M}_{i}(\underline{X}_{t_i})-\Xi^N_{i}(\underline{X}_{t_i})\middle|\mathcal{F}_0^M, X_{t_i}\right]\right)\middle|\mathcal{F}_0^M, X_{t_i}\right]^2\right]\\
&\le \mathcal{D} E\left[E\left[\left(E\left[\Xi^{N,M}_{i}(\underline{X}_{t_i})-\Xi^N_{i}(\underline{X}_{t_i})\middle|\mathcal{F}_0^M, \mathcal{F}_{t_{i+1}}\right]\right.\right.\right.\\
&\quad \left.\left.\left.-E\left[\Xi^{N,M}_{i}(\underline{X}_{t_i})-\Xi^N_{i}(\underline{X}_{t_i})\middle|\mathcal{F}_0^M, X_{t_i}\right]\right)^2\middle|\mathcal{F}_0^M, X_{t_i}\right]\right]\\
&\le  \mathcal{D} E\left[E\left[E\left[\Xi^{N,M}_{i}(\underline{X}_{t_i})-\Xi^N_{i}(\underline{X}_{t_i})\middle|\mathcal{F}_0^M, \mathcal{F}_{t_{i+1}}\right]^2\middle|\mathcal{F}_0^M,X_{t_i}\right]\right.\\
&\quad -\left.\mathcal{D} E\left[\Xi^{N,M}_{i}(\underline{X}_{t_i})-\Xi^N_{i}(\underline{X}_{t_i})\middle|\mathcal{F}_0^M, X_{t_i}\right]^2\right].
\end{align}
For the next step we have to distinguish between the time points again. If $t_i$ and $t_{i+1}$ are in the same segment, i.e., at all time points  $t_i$ such that $t_{i+1}\not\in \overline{\pi}_\alpha$, it holds $\tau_\alpha(i)=\tau_\alpha(i+1)$ and we get for a $\Gamma>0$ which will be specified later that
\begin{align}
&\Delta E\left[\left(\xi^{z}_i(X_{t_i})\right)^2\right]\\
&\le \mathcal{D}E\left[E\left[\overline{q}_{\tau_\alpha(i)}(X_{t_{\tau_\alpha(i)}})-q^{N,M}_{\tau_\alpha(i)}(X_{t_{\tau_\alpha(i)}})+\sum_{j=i+1}^{\tau_\alpha(i)}\int_{t_j}^{t_{j+1}}f(s,X_s,Y_s,Z_s)\right.\right.\\
&\quad \left.\left. -f\left(t_j,X_{t_j},q_j^{N,M},z_j^{N,M}\right)ds\middle|\mathcal{F}_0^M,\mathcal{F}_{t_{i+1}}\right]^2\right]-\mathcal{D} E\left[E\left[\Xi^{N,M}_{i}(\underline{X}_i)-\Xi^N_{i}(\underline{X}_i)\middle|\mathcal{F}_0^M, X_{t_i}\right]^2\right]\\
&\le (1+\Gamma\Delta)\mathcal{D}E\left[\left(\xi^{q}_{i+1}(X_{t_{i+1}})\right)^2\right]\\
&\quad +\mathcal{D}(1+\frac{1}{\Delta\Gamma})E\left[E\left[\int_{t_{i+1}}^{t_{i+2}}f(s,X_s,Y_s,Z_s)-f(t_{i+1},X_{t_{i+1}},q_{i+1}^{N,M},z_{i+1}^{N,M})ds\middle|\F_0^M,X_{t_i}\right]^2\right]\\
&\quad -\mathcal{D}E\left[\left(\xi^{q}_i(X_{t_i})\right)^2\right].
\end{align}
Like in the derivation of the error of the approximation of  $\overline{q}^N$, the last inequality in the calculations above does not hold true when considering $E[\xi^{z}_i(X_{t_i})^2]$ at time points $t_i$ at the end of a segment, i.e., $t_i$ such that $t_{i+1}\in \overline{\pi}_\alpha$. However, by adding and subtracting a multiple of $E[(\xi^{q}_{i+1}(X_{t_{i+1}}))^2]$, we again get the similar bound  
\begin{align}
&\Delta E\left[\left(\xi^{z}_i(X_{t_i})\right)^2\right]\\
&\le \mathcal{D}(1+\Gamma\Delta)(1+\kappa)(1+\epsilon)E\left[\left(\xi^{q}_{i+1}(X_{t_{i+1}})\right)^2\right]\\
&\quad +\mathcal{D}(1+\frac{1}{\Delta\Gamma})E\left[E\left[\int_{t_{i+1}}^{t_{i+2}}f(s,X_s,Y_s,Z_s)-f(t_{i+1},X_{t_{i+1}},q_{i+1}^{N,M},z_{i+1}^{N,M})ds\middle|\F_0^M,X_{t_i}\right]^2\right]\\
&\quad -\mathcal{D}E\left[\left(\xi^{q}_{i}(X_{t_i})\right)^2\right]\\
&\quad +\mathcal{D}(1+\Gamma\Delta)\left(E\left[\big\|\overline{q}^N_{i+1}-q_{i+1}^{N,M}\big\|^2_{i+1,\infty}\right]-(1+\kappa)(1+\epsilon)E\left[\left(\xi^{q}_{i+1}(X_{t_{i+1}})\right)^2\right]\right)_+
\end{align}
for these time points, with an additional error term that depends on the approximation of $\overline{q}^N$ at the next time point $t_{i+1}$. As argued in the analysis of the approximation of $\overline{q}^N$, this term results from the single use of an ODP step in the discretization scheme that is used to connect two time segments.\\
Next we want to derive a global error bound for the approximation of $\overline{z}^N$. Since $\overline{z}^N$ appears in the discretization scheme only as argument of the driver,  we state this term as an averaged sum over the time steps rather than the maximum. 
\begin{lem}\label{lem:Z}
Let $\Gamma$ be a positive constants and set $\lambda_i:=(1+\Gamma\Delta)^i((1+N^{\alpha-1})^{2|\{j\le i: t_j\in \overline{\pi}_\alpha\}|}$ for $i\in  \{0,\ldots,N-1\}$. Then
\begin{align}
&\sum_{i=0}^{N-1}\Delta \lambda_iE\left[\left(\xi^{z}_i(X_{t_i})\right)^2\right]\\
&\le \mathcal{D}\sum_{i=0}^{N-2}\lambda_{i}(1+\frac{1}{\Delta\Gamma})E\left[E\left[\int_{t_{i+1}}^{t_{i+2}}f(s,X_s,Y_s,Z_s)-f(t_{i+1},X_{t_{i+1}},q_{i+1}^{N,M},z_{i+1}^{N,M})ds\middle|\F_0^M,X_{t_i}\right]^2\right]\\
&\quad +\lambda_N\mathcal{D}\lceil N^{1-\alpha}\rceil\max_{j\in \mathcal{I}}\left((1+N^{\alpha-1})\left(\inf_{\psi \in \mathcal{K}_{q,j}}E\left[\big|\psi(X_{t_j}-\overline{q}^N_i(X_{t_j})\big|^2\right]+(1+N^{1-\alpha})\frac{C^2_{q,j}K_{q,j}}{M_j}\right)\right.\\
&\quad \left.+\frac{N^{1-\alpha}C_1K_{q,j}\log(C_2M_{j})}{M_{j}}\right).
\end{align}
\end{lem}
\begin{proof}
First, note that $E[(\xi^{z}_{N-1}(X_{t_{N-1}}))^2]=0$ by definition since $\Xi^{{N,M}}_{N-1}=\Xi^N_{N-1}$. Then, by plugging in the estimate for $\Delta E[(\xi^{z}_i(X_{t_i}))^2]$ from the analysis of the error propagation where we choose $\epsilon=\kappa=N^{\alpha-1}$ for all $i$ and summing up we get
\begin{align}
&\sum_{i=0}^{N-2}\Delta \lambda_{i}E\left[\left(\xi^{z}_i(X_{t_i})\right)^2\right]\\
&\le \mathcal{D}\sum_{i=0}^{N-2} \Bigg(\lambda_{i+1}E\left[\left(\xi^{q}_{i+1}(X_{t_{i+1}})\right)^2\right]\\
&\quad +\lambda_i(1+\frac{1}{\Delta\Gamma}) E\left[E\left[\int_{t_{i+1}}^{t_{i+2}}f(s,X_s,Y_s,Z_s)-f(t_{i+1},X_{t_{i+1}},q_{i+1}^{N,M},z_{i+1}^{N,M})ds\middle|\F_0^M,X_{t_i}\right]^2\right]\\
&\quad -\lambda_iE\left[\left(\xi^{q}_{i}(X_{t_i})\right)^2\right]\\\
&\quad +\lambda_i(1+\Gamma\Delta)\left(E\left[\big\|\overline{q}^N_{i+1}-q^{N,M}_{i+1}\big\|^2_{i+1,\infty}\right]-(1+N^{\alpha-1})^2E\left[\left(\xi^{q}_{i+1}(X_{t_{i+1}})\right)^2\right]\right)_+\mathds{1}_{\mathcal{I}}(i+1)\Bigg)\\
&\le \mathcal{D}\sum_{i=0}^{N-2}\lambda_{i}(1+\frac{1}{\Delta\Gamma}) E\left[E\left[\int_{t_{i+1}}^{t_{i+2}}f(s,X_s,Y_s,Z_s)-f(t_{i+1},X_{t_{i+1}},q_{i+1}^{N,M},z_{i+1}^{N,M})ds\middle|\F_0^M,X_{t_i}\right]^2\right]\\
&\quad +\mathcal{D}\lambda_N\lceil N^{1-\alpha}\rceil\sup_{j \in \mathcal{I}}\left(\inf_{\psi \in \mathcal{K}_{q,j}}E\left[\big|\psi(X_{t_j}-\overline{q}^N_j(X_{t_j})\big|^2\right] -(1+N^{\alpha-1})^2E\left[\left(\xi^{q}_{j}(X_{t_{j}})\right)^2\right]\right)_+.
\end{align}
Plugging in the estimate for  $E[\|\overline{q}^N_{i}-q_{i}^{N,M}\|^2_{i,\infty}]$ from \eqref{eq:q3} derived in the proof of Lemma \ref{lem:Q} finishes the proof.
\end{proof}
\subsection{Bounds for the driver-dependent terms}
In this section, we derive a bound for the sum of the terms 
\begin{align}
E\left[E\left[\int_{t_{i+1}}^{t_{i+2}}f\left(s,X_s,Y_s,Z_s\right)-f\left(t_{i+1},X_{t_{i+1}},q_{i+1}^{N,M},z_{i+1}^{N,M}\right)ds\middle|\F^M,X_{t_i}\right]^2\right]
\end{align}
over the time steps that appears in the bounds of Lemma \ref{lem:Q} and \ref{lem:Z}. For this purpose, note that for any $i\in \{0,\ldots, N-1\}$, it holds by Fubini's theorem 
\begin{align}
&E\left[E\left[\int_{t_{i+1}}^{t_{i+2}}f\left(s,X_s,Y_s,Z_s\right)-f\left(t_{i+1},X_{t_{i+1}},q_{i+1}^{N,M},z_{i+1}^{N,M}\right)ds\middle|\mathcal{F}_0^M, X_{t_{i}}\right]^2\right]\\
&\le E\Bigg[\bigg(\int_{t_{i+1}}^{t_{i+2}}E_i\left[f\left(s,X_s,Y_s,Z_s\right)-f\left(t_{i+1},X_{t_{i+1}},\overline{q}^N_{i+1},\overline{z}^N_{i+1}\right)\right]\\
&\quad +E\left[f\left(t_{i+1},X_{t_{i+1}},\overline{q}^N_{i+1},\overline{z}^N_{i+1}\right)-f\left(t_{i+1},X_{t_{i+1}},q^{N,M}_{i+1},z^{N,M}_{i+1}\right)\middle|\mathcal{F}_0^M,X_{t_i}\right]ds\bigg)^2\Bigg]\\
&\le 2E\left[\left(\int_{t_{i+1}}^{t_{i+2}}E_i\left[f\left(s,X_s,Y_s,Z_s\right)-f\left(t_{i+1},X_{t_{i+1}},\overline{q}^N_{i+1},\overline{z}^N_{i+1}\right)\right]ds\right)^2\right]\\
&\quad +2\Delta^2E\left[E\left[f\left(t_{i+1},X_{t_{i+1}},\overline{q}^N_{i+1},\overline{z}^N_{i+1}\right)-f\left(t_{i+1},X_{t_{i+1}},q^{N,M}_{i+1},z^{N,M}_{i+1}\right)\middle|\mathcal{F}_0^M,X_{t_i}\right]^2\right].
\end{align}
Then, by the Lipschitz assumption on $f$ we get
\begin{align}
&E\left[E\left[f\left(t_{i+1},X_{t_{i+1}},\overline{q}^N_{i+1},\overline{z}^N_{i+1}\right)-f\left(t_{i+1},X_{t_{i+1}},q^{N,M}_{i+1},z^{N,M}_{i+1}\right)\middle|\mathcal{F}_0^M,X_{t_i}\right]^2\right]\\
&\le 2L_f^2E\left[E\left[\big|\overline{q}^N_{i+1}(X_{t_{i+1}})-q_{i+1}^{N,M}(X_{t_{i+1}})\big|\middle|\mathcal{F}^M_0,X_{t_i}\right]^2\right.\\
&\quad \left.+E\left[\big|\overline{z}^N_{i+1}(X_{t_{i+1}})-z_{i+1}^{N,M}(X_{t_{i+1}})\big|\middle|\mathcal{F}_0^M,X_{t_i}\right]^2\right]\\
&\leq 2L_f^2\left(E\left[\big\|\overline{q}^N_{i+1}-q_{i+1}^{N,M}\big\|_{i+1,\infty}^2\right]+E\left[\big\|\overline{z}^N_{i+1}-z_{i+1}^{N,M}\big\|_{i+1,\infty}^2\right]\right).\\
\end{align}
Hence it holds
\begin{align}
&\sum_{i=0}^{N-2}\lambda_{i}E\left[E\left[\int_{t_{i+1}}^{t_{i+2}}f\left(s,X_s,Y_s,Z_s\right)-f\left(t_{i+1},X_{t_{i+1}},q_{i+1}^{N,M},z_{i+1}^{N,M}\right)ds\middle|\F_0^M,X_{t_i}\right]^2\right]\\
&\le \sum_{i=0}^{N-2}\lambda_{i}2E\left[\left(\int_{t_{i+1}}^{t_{i+2}}E_i\left[f\left(s,X_s,Y_s,Z_s\right)-f\left(t_{i+1},X_{t_{i+1}},\overline{q}^N_{i+1},\overline{z}^N_{i+1}\right)\right]ds\right)^2\right]\\
&\quad +\sum_{i=0}^{N-2}\lambda_{i}4\Delta^2L_f^2\left(E\left[\big\|\overline{q}^N_{i+1}-q_{i+1}^{N,M}\big\|_{i+1,\infty}^2\right]+E\left[\big\|\overline{z}^N_{i+1}-z_{i+1}^{N,M}\big\|_{i+1,\infty}^2\right]\right)\\
&\le \sum_{i=0}^{N-2}\lambda_{i}2E\left[\left(\int_{t_{i+1}}^{t_{i+2}}E_i\left[f\left(s,X_s,Y_s,Z_s\right)-f\left(t_{i+1},X_{t_{i+1}},\overline{q}^N_{i+1},\overline{z}^N_{i+1}\right)\right]ds\right)^2\right]\\
&\quad +4\Delta L_f^2T\max_{0\le i\le N-1}\lambda_iE\left[\big\|\overline{q}^N_{i+1}-q_{i+1}^{N,M}\big\|_{i+1,\infty}^2\right]+4\Delta L_f^2\sum_{i=1}^{N-1}\lambda_i\Delta E\left[\big\|\overline{z}^N_{i+1}-z_{i+1}^{N,M}\big\|_{i+1,\infty}^2\right]\\
&\le2\lambda_N\mathcal{R}^N+4\Delta L_f^2T\max_{0\le i\le N-1}\lambda_iE\left[\big\|\overline{q}^N_{i+1}-q_{i+1}^{N,M}\big\|_{i+1,\infty}^2\right]+4\Delta L_f^2\sum_{i=1}^{N-1}\lambda_i\Delta E\left[\big\|\overline{z}^N_{i+1}-z_{i+1}^{N,M}\big\|_{i+1,\infty}^2\right]
\end{align}
with 
\begin{align}
\mathcal{R}^N:=\sum_{i=0}^{N-2}E\left[\left(\int_{t_{i+1}}^{t_{i+2}}E_i\left[f\left(s,X_s,Y_s,Z_s\right)-f\left(t_{i+1},X_{t_{i+1}},\overline{q}^N_{i+1},\overline{z}^N_{i+1}\right)\right]ds\right)^2\right]
\end{align}
as defined in Theorem \ref{thm:error1}. The term $\mathcal{R}^N$ does not depend on our approximation of the BSDE but only on the real solution $Y,Z$, the semi-continuous versions $\overline{q}^N,\overline{z}^N$ and the solution of the forward SDE $X$. It can be bounded in different ways depending on the regularity of these functions, which leads to the different  bounds of the total quadratic error in Theorem \ref{thm:error2} and \ref{thm:error3}.
\subsection{Final error bounds}
Using the bounds derived throughout this section, we are now ready to proof Theorem \ref{thm:error1}.
\begin{proof}
Proof of Theorem \ref{thm:error1}:\\
In the following calculations, $c$ denotes a positive constant that does not depend on $N$ and may change from line to line. First, we can write the quadratic error as
\begin{align}
&\max_{0\le i \le N-1}E\left[\big|\overline{q}^N_i(X_{t_i})-q_i^{N,M}(X_{t_i})\big|^2\right]+\sum_{i=0}^{N-1}\Delta E\left[\big|\overline{z}^N(X_{t_i})-z_i^{N,M}(X_{t_i})\big|^2\right]\\
&\le \max_{0\le i\le N-1}\lambda_iE\left[\big|\overline{q}^N_i(X_{t_i})-q_i^{N,M}(X_{t_i})\big|^2\right]+\sum_{i=0}^{N-2}\Delta\lambda_iE\left[\big|\overline{z}^N(X_{t_i})-z_i^{N,M}(X_{t_i})\big|^2\right]\\
&=\max_{0\le i \le N-1}\lambda_iE\left[\big\|\overline{q}^N_i-q_i^{N,M}\big\|^2_{i,\infty}\right]+\sum_{i=0}^{N-1}\Delta\lambda_iE\left[\big\|\overline{z}^N_i-z_i^{N,M}\big\|^2_{i,\infty}\right].
\label{eq:finalerror}
\end{align}
We can now estimate this term by Lemma \ref{lem:samplechange} with $\epsilon=1$ as
\begin{align}
&\max_{0\le i\le N-1}\lambda_iE\left[\big\|\overline{q}^N_i-q_i^{N,M}\big\|^2_{i,\infty}\right]+\sum_{i=0}^{N-1}\Delta\lambda_iE\left[\big\|\overline{z}^N_i-z_i^{N,M}\big\|^2_{i,\infty}\right]\\
&\le\max_{0\le i\le N-1}\left(2\lambda_iE\left[\big\|\overline{q}^N_i-q_i^{N,M}\big\|^2_{i,M}\right]+\lambda_i \frac{C_1K_{q,i}\log(C_2M_i)}{M_i}\right)\\
&\quad +\sum_{i=0}^{N-2}2\Delta \lambda_iE\left[\big\|\overline{z}^N_i-z_i^{N,M}\big\|^2_{i,M}\right]+\Delta\lambda_i\frac{C_1K_{z,i}\log(C_2M_i)}{\Delta M_i}.
\end{align}
By the inequalities \eqref{eq:q2} and \eqref{eq:z3} we then get with the choice $\kappa=1$
\begin{align}
&\max_{0\le i\le N-1}\lambda_iE\left[\big\|\overline{q}^N_i-q_i^{N,M}\big\|^2_{i,\infty}\right]+\sum_{i=0}^{N-1}\Delta\lambda_iE\left[\big\|\overline{z}^N_i-z_i^{N,M}\big\|^2_{i,\infty}\right]\\
&\le\lambda_N\max_{0\le i\le N-1}\left(4\frac{C_{q,i}K_{q,i}}{M_i}+2\underset{\psi\in \mathcal{K}_{q,i}}{\inf}E\left[\left|\overline{q}^N_i(X_{t_i})-\psi(X_{t_i})\right|^2\right]+\frac{C_1K_{q,i}\log(C_2M_i)}{M_i}\right)\\
&\quad +\sum_{i=0}^{N-1}\Delta\lambda_i  \left(4\frac{C_{q,i}^2K_{z,i}}{\Delta M_i}+2\underset{\psi\in \mathcal{K}_{z,i}}{\inf}E\left[\left|\overline{z}^N_i(X_{t_i})-\psi(X_{t_i})\right|^2\right]+\frac{C_1K_{z,i}\log(C_2M_i)}{\Delta M_i}\right)\\
&\quad +4\left(\max_{0\le i\le N-1}\lambda_iE\left[\left(\xi^q_i(X_{t_i})\right)^2\right]+\sum_{i=0}^{N-1}\Delta\lambda_iE\left[\left(\xi^{z}_i(X_{t_i})\right)^2\right]\right).\\
\end{align}
Now the bounds in Lemma \ref{lem:Q} and Lemma \ref{lem:Z} yield 
\begin{align}
&\max_{0\le i \le N-1}\lambda_iE\left[\left(\xi^q_i(X_{t_i})\right)^2\right]+\sum_{i=0}^{N-1}\Delta\lambda_iE\left[\left(\xi^{z}_i(X_{t_i})\right)^2\right]\\
&\le \sum_{i=0}^{N-2}(1+\frac{1}{\Delta\Gamma})\lambda_iE\left[E\left[\int_{t_{i+1}}^{t_{i+2}}f\left(s,X_s,Y_s,Z_s\right)-f\left(t_{i+1},X_{t_{i+1}},q^{N,M}_{i+1},z^{N,M}_{i+1}\right)ds\middle|\F_0^M,X_{t_i}\right]^2\right]\\
&\quad + \lceil N^{1-\alpha}\rceil\lambda_N\max_{j\in \mathcal{I}}\Bigg((1+N^{\alpha-1})\left(\underset{\psi\in \mathcal{K}_{q,i}}{\inf}E\left[\left|\overline{q}^N_i(X_{t_i})-\psi(X_{t_i})\right|^2\right]+(1+N^{1-\alpha})\frac{C_{q,i}^2K_{q,j}}{M_j}\right)\\
&\quad +\frac{N^{1-\alpha}C_1K_{q,j}\log(C_2M_{j})}{M_{j}}\Bigg)\\
&\quad +\mathcal{D}\sum_{i=0}^{N-2}\lambda_{i}(1+\frac{1}{\Delta\Gamma}) E\left[E\left[\int_{t_{i+1}}^{t_{i+2}}f\left(s,X_s,Y_s,Z_s\right)-f\left(t_{i+1},X_{t_{i+1}},q^{N,M}_{i+1},z^{N,M}_{i+1}\right)ds\middle|\F_0^M,X_{t_i}\right]^2\right]\\
&\quad +\lambda_N\mathcal{D}\lceil N^{1-\alpha}\rceil\max_{j\in \mathcal{I}}\Bigg((1+N^{\alpha-1})\left(\underset{\psi\in \mathcal{K}_{q,i}}{\inf}E\left[\left|\overline{q}^N_i(X_{t_i})-\psi(X_{t_i})\right|^2\right]+(1+N^{1-\alpha})\frac{C_{q,i}^2K_{q,j}}{M_j}\right)\\
&\quad +\frac{N^{1-\alpha}C_1K_{q,j}\log(C_2M_{j})}{M_{j}}\Bigg)\\
&= \sum_{i=0}^{N-2}(1+\mathcal{D})(1+\frac{1}{\Delta\Gamma})\lambda_iE\left[E\left[\int_{t_{i+1}}^{t_{i+2}}f\left(s,X_s,Y_s,Z_s\right)-f\left(t_{i+1},X_{t_{i+1}},q^{N,M}_{i+1},z^{N,M}_{i+1}\right)ds\middle|\F_0^M,X_{t_i}\right]^2\right]\\
&\quad +(1+\mathcal{D})\lceil N^{1-\alpha}\rceil \lambda_N\max_{j\in \mathcal{I}}\Bigg((1+N^{\alpha-1})\left(\underset{\psi\in \mathcal{K}_{q,i}}{\inf}E\left[\left|\overline{q}^N_i(X_{t_i})-\psi(X_{t_i})\right|^2\right]+(1+N^{1-\alpha})\frac{C_{q,i}^2K_{q,j}}{M_j}\right)\\
&\quad +\frac{N^{1-\alpha}C_1K_{q,j}\log(C_2M_{j})}{M_{j}}\Bigg).\\
\end{align}
By plugging in the bounds derived in Section 6.4, we can estimate this term further as
\begin{align}
&\max_{0\le i \le N-1}\lambda_iE\left[\left(\xi^q_i(X_{t_i})\right)^2\right]+\sum_{i=0}^{N-1}\Delta\lambda_iE\left[\left(\xi^{z}_i(X_{t_i})\right)^2\right]\\
&\le \sum_{i=0}^{N-2}(1+\mathcal{D})(1+\frac{1}{\Delta\Gamma})\lambda_iE\left[E\left[\int_{t_{i+1}}^{t_{i+2}}f\left(s,X_s,Y_s,Z_s\right)-f\left(t_{i+1},X_{t_{i+1}},q^{N,M}_{i+1},z^{N,M}_{i+1}\right)ds\middle|\F_0^M,X_{t_i}\right]^2\right]\\
&\quad +(1+\mathcal{D})\lceil N^{1-\alpha}\rceil \lambda_N\max_{j\in \mathcal{I}}\Bigg((1+N^{\alpha-1})\left(\underset{\psi\in \mathcal{K}_{q,i}}{\inf}E\left[\left|\overline{q}^N_i(X_{t_i})-\psi(X_{t_i})\right|^2\right]+(1+N^{1-\alpha})\frac{C_{q,i}^2K_{q,j}}{M_j}\right)\\
&\quad +\frac{N^{1-\alpha}C_1K_{q,j}\log(C_2M_{j})}{M_{j}}\Bigg)\\
&\le \left[(\Delta+\frac{1}{\Gamma})(1+\mathcal{D})\right] 4(T\vee 1)L_f^2\left(\max_{0\le i\le N-1}\lambda_iE\left[\big\|\overline{q}^N_{i}-q^{N,M}_i\big\|^2_{i,\infty}\right]+\sum_{i=0}^{N-2}\Delta \lambda_i E\left[\big\|\overline{z}^N_i-z^{N,M}_{i}\big\|^2_{i,\infty}\right]\right)\\
&\quad +  \left[(\Delta+\frac{1}{\Gamma})(1+\mathcal{D})\right]\Delta^{-1}\lambda_N\mathcal{R}^N\\
&\quad +\lceil N^{1-\alpha}\rceil \lambda_N(1+\mathcal{D})\max_{j\in \mathcal{I}}\Bigg((1+N^{\alpha-1})\left(\underset{\psi\in \mathcal{K}_{q,i}}{\inf}E\left[\left|\overline{q}^N_i(X_{t_i})-\psi(X_{t_i})\right|^2\right]+(1+N^{1-\alpha})\frac{C_{q,j}^2K_{q,j}}{M_j}\right)\\
& \quad+\frac{N^{1-\alpha}C_1K_{q,j}\log(C_2M_{j})}{M_{j}}\Bigg).\\
\end{align}
Now, assuming that  $N$ and $\Gamma$ are sufficiently large such that $ [(\Delta+\frac{1}{\Gamma})(1+\mathcal{D})]16L^2(T\vee 1)\le \frac{1}{2}$, we have
\begin{align}
&\max_{0\le i\le N-1}E\left[\big\|\overline{q}^N_i-q_i^{N,M}\big\|^2_{i,\infty}\right]\lambda_i+\sum_{i=0}^{N-1}\lambda_iE\left[\big\|\overline{z}^N_i-z_i^{N,M}\big\|^2_{i,\infty}\right]\Delta\\
&\le \frac{1}{2}\left(\max_{0\le i\le N-1}E\left[\big\|\overline{q}^N_i-q_i^{N,M}\big\|^2_{i,\infty}\right]\lambda_i+\sum_{i=0}^{N-1}\lambda_iE\left[\big\|\overline{z}^N_i-z_i^{N,M}\big\|^2_{i,\infty}\right]\Delta\right)\\
&+c\lambda_N\lceil N^{1-\alpha}\rceil \max_{j\in \mathcal{I}}\Bigg((1+N^{\alpha-1})\left(\underset{\psi\in \mathcal{K}_{q,i}}{\inf}E\left[\left|\overline{q}^N_i(X_{t_i})-\psi(X_{t_i})\right|^2\right]+(1+N^{1-\alpha})\frac{C_{q,j}^2K_{q,j}}{M_j}\right)\\
& \quad+\frac{N^{1-\alpha}C_1K_{q,j}\log(C_2M_{j})}{M_{j}}\Bigg)\\
&\quad +c\lambda_N \max_{0\le i\le N-1}\Bigg(\underset{\psi\in \mathcal{K}_{q,i}}{\inf}E\left[\left|\overline{q}^N_i(X_{t_i})-\psi(X_{t_i})\right|^2\right]+\frac{K_{q,i}}{M_i}+\frac{K_{q,i}\log(M_i)}{M_i}\\
&\quad +\underset{\psi\in \mathcal{K}_{z,i}}{\inf}E\left[\left|\overline{z}^N_i(X_{t_i})-\psi(X_{t_i})\right|^2\right]+\frac{K_{z,i}}{\Delta M_i}+\frac{K_{z,i}\log(M_i)}{\Delta M_i}\Bigg)\\
&\quad +c\Delta^{-1}\lambda_N\mathcal{R}^N.
\end{align}
Considering that $\lambda_N$ is bounded by a constant independent of $N$, since
\begin{align}
\lambda_N&=\left(1+\frac{T\Gamma}{N}\right)^N\left(1+N^{\alpha-1}\right)^{2\lceil N^{1-\alpha}\rceil}\\
&\le e^{T\Gamma N^{-1}N}e^{2\lceil N^{1-\alpha}\rceil N^{\alpha-1}}=e^{T\Gamma+4},
\end{align}
this implies
\begin{align}
&\max_{0\le i\le N-1}E\left[\big|\overline{q}^N_i(X_{t_i})-q_i^{N,M}(X_{t_i})\big|^2\right]+\sum_{i=0}^{N-2}\Delta E\left[\big|\overline{z}_i^N(X_{t_i})-z_i^{N,M}(X_{t_i})\big|^2\right]\\
&\le c\max_{i\in \mathcal{I}}\left(N^{1-\alpha} \inf_{\psi \in \mathcal{K}_{q,i}} E\left[\big|\psi(X_{t_i})-\overline{q}^N_i(X_{t_i})\big|^2\right] + N^{2-2\alpha}\frac{K_{q,i}}{M_i}+ N^{2-2\alpha}\frac{K_{q,i}\log(C_2M_i)}{M_i}\right)\\
&\quad +c\max_{0\le i\le N-1}\bigg(\inf_{\psi \in \mathcal{K}_{q,i}} E\left[\big|\psi(X_{t_i})-\overline{q}^N_i(X_{t_i})\big|^2\right] + \inf_{\psi \in \mathcal{K}_{z,i}} E\left[\big|\psi(X_{t_i})-\overline{z}^N_i(X_{t_i})\big|^2\right]\\
&\quad+\frac{K_{q,i}}{M_i}+N\frac{K_{z,i}}{M_i}+\frac{K_{q,i}\log(C_2M_i)}{M_i}+ N \frac{K_{z,i}\log(C_2M_i)}{M_i}\bigg)\\
&  \quad+cN\mathcal{R}^N
\end{align}
and finishes the proof.
\end{proof}
\noindent

\end{document}